\documentclass[11pt]{amsart}

\usepackage[utf8x]{inputenc}
\usepackage[english]{babel}

\usepackage[pdftex]{graphicx}
\usepackage[margin=2.8cm]{geometry}

\usepackage{amsfonts}
\usepackage{amsmath}
\usepackage{amssymb}
\usepackage{amsthm}
\usepackage{dsfont}
\usepackage{tikz}
\usepackage{pgfplots}

\usepackage{float}

\usepackage[T1]{fontenc}

\usepackage{paralist}

\usepackage[colorlinks,link color=blue, cite color=blue,pagebackref,pdftex]{hyperref}

\newcommand\Defn[1]{\textbf{\color{black}#1}}

\newcommand\eps{\varepsilon}
\renewcommand\emptyset{\varnothing}
\renewcommand\phi{\varphi}
               
\newcommand\Z{\mathbb{Z}}               
\newcommand\Q{\mathbb{Q}}               
\newcommand\R{\mathbb{R}}

\newcommand{\hosymbol}{
\begin{tikzpicture}
\draw[very thin] (0,0) circle (0.05cm);
\filldraw[black] (0,0)--(0,-0.05cm) arc (-90:90:0.05cm) -- (0,0);
\end{tikzpicture}
}

\newcommand\Int[2]{{#1}^{\cap{#2}}}
\newcommand\Mink[2]{{#1}^{+{#2}}}

\newcommand\HO[1]{#1^{\hosymbol}}
\newcommand\half[2]{\mathrm{H}_{#1}{#2}}
\newcommand\ihalf[2]{\mathrm{H}^{#1}{#2}}
\newcommand\Triang{\mathcal{T}}

\newcommand\LL{\Lambda}
\newcommand\PolL{\mathcal{P}(\LL)}
\newcommand\Pol[1]{\mathcal{P}(#1)}
\newcommand\Vals{\mathcal{V}}

\newcommand\vol{V}
\newcommand\dvol{\mathrm{E}}

\newcommand\wSys{{\boldsymbol{\nu}}}
    \newcommand\Hom[1]{\widehat{#1}}
\newcommand\Qcomp[2]{\mathcal{C}_{#1}(#2)}
\newcommand\Euler{\tilde{\chi}}%
\newcommand\ValsCP{\Vals_{\text{\sc CP}}}
\newcommand\ValsWCP{\Vals_{\text{\sc WCP}}}
\newcommand\bVals{\overline{\mathcal{V}}}
\newcommand\bValsCP{\bVals_{\text{\sc CP}}}
\newcommand\bValsM{\bVals_{\text{\sc M}}}
\newcommand\bValsN{\bVals_{\text{\sc +}}}

\DeclareMathOperator*{\relint}{relint}
\DeclareMathOperator*{\aff}{aff}
\DeclareMathOperator*{\conv}{conv}

\newtheorem{thm}{Theorem}[section]
\newtheorem*{thm*}{Theorem}
\newtheorem{cor}[thm]{Corollary}
\newtheorem{lem}[thm]{Lemma}
\newtheorem{prop}[thm]{Proposition}
\newtheorem{conj}{Conjecture}
\newtheorem{quest}{Question}

\theoremstyle{definition}

\newtheorem{example}[thm]{Example}

\title[Combinatorial positivity and a discrete Hadwiger theorem]{Combinatorial positivity of translation-invariant valuations and a
discrete Hadwiger theorem}

\author{Katharina Jochemko}
\address{Department of Mathematics, %
Royal Institute of Technology (KTH), %
Sweden}
\email{jochemko@kth.se}
\author{Raman Sanyal}
\address{FB 12 - Institut f\"ur Mathematik, %
Goethe-Universit\"at Frankfurt, %
Germany}
\email{sanyal@math.uni-frankfurt.de}



\date{\today}

\begin{document}

\maketitle

\begin{abstract}
We introduce the notion of combinatorial positivity of
translation-invariant valuations on convex polytopes that extends the
nonnegativity of Ehrhart $h^*$-vectors. We give a surprisingly simple
characterization of combinatorially positive valuations that implies
Stanley's nonnegativity and monotonicity of $h^*$-vectors and generalizes
work of Beck et al.\  (2010) from solid-angle polynomials to all
translation-invariant simple valuations.  For general polytopes, this
yields a new characterization of the volume as the unique combinatorially
positive valuation up to scaling. For lattice polytopes our results extend
work of Betke--Kneser (1985) and give a discrete Hadwiger theorem: There
is essentially a unique combinatorially-positive basis for the space of
lattice-invariant valuations.  As byproducts of our investigations, we
prove a multivariate Ehrhart--Macdonald reciprocity and we show
universality of weight valuations studied in Beck et al.\ (2010).
\end{abstract}

\section{Introduction}\label{sec:intro}

A celebrated result of Ehrhart~\cite{ehrhart} states that for a convex lattice
polytope $P = \conv(V)$, $V \subset \Z^d$, the function $ \dvol_P(n)  :=  | nP
\cap \Z^d |$  agrees with a polynomial---the \emph{Ehrhart polynomial} of $P$.
More precisely, there are unique $h^*_0,h^*_1,\dots,h^*_r \in \Z$ with $r = \dim P$ such
that
\begin{equation}\label{eqn:hE-vec}
    \dvol_P(n) \ = \ h^*_0 \binom{n+r}{r} + h^*_1 \binom{n+r-1}{r} + \cdots 
                 + h^*_r \binom{n}{r}
\end{equation}
for all $n \in \Z_{\ge0}$. In the language of generating functions this states
\[
    \sum_{n \ge 0} \dvol_P(n) z^n \ = \ \frac{h^*_0 + h^*_1 z + \cdots + h^*_r
    z^r}{(1-z)^{r+1}}.
\]
Ehrhart polynomials miraculously occur in many areas such as
combinatorics~\cite{BZ, breuerS, stanley74}, commutative algebra and algebraic
geometry~\cite{MillerSturmfels}, and representation
theory~\cite{BerensteinZelevinsky, deLMc}. The question which polynomials can
occur as Ehrhart polynomials is well-studied~\cite{BdLDPS, BetkeMcMullen,
HenkTagami, Stapledon} but wide open. Groundbreaking contributions to that
question are two theorems of Stanley~\cite{Stan80,Stan93}.  Define the
\Defn{$\boldsymbol{h^*}$-vector}\footnote{Also called the $\delta$-vector or
Ehrhart $h$-vector.} of $P$ as $h^*(P) := (h^*_0,h^*_1,\dots,h^*_d)$ where we
set $h^*_i =
0$ for $i > \dim P$. Stanley showed that $h^*$-vectors of lattice polytopes
satisfy a nonnegativity and monotonicity property: If $P \subseteq Q$ are
lattice polytopes, then 
\[
    0 \ \le \ h_i^*(P) \ \le \ h_i^*(Q)
\]
for all $i = 0,\dots, d$. 

McMullen~\cite{mcmullenEuler} generalized Ehrhart's result to
translation-invariant valuations. For now, let \mbox{$\LL \in \{\Z^d,\R^d\}$}
and $\Pol{\LL}$ be the collection of polytopes with vertices in $\LL$. A map
$\phi : \Pol{\LL} \rightarrow \R$ is a \Defn{translation-invariant valuation}
if $\phi(\emptyset) = 0$ and $\phi(P \cup Q) + \phi(P \cap Q) = \phi(P) +
\phi(Q)$ whenever $P, Q, P \cup Q, P\cap Q \in \Pol{\LL}$, and $\phi(t + P) =
\phi(P)$ for all $t \in \LL$. Valuations are a cornerstone of modern discrete
and convex geometry.  The study of valuations invariant under the action of a
group of transformations is an area of active research with beautiful
connections to algebra and combinatorics;
see~\cite{RotaKlain,mcmullenAlgComb}.  For example, for $\LL = \Z^d$, the
\Defn{discrete volume} $\dvol(P) := | P \cap \LL|$ is clearly a
translation-invariant valuation. 

McMullen showed that for every $r$-dimensional polytope $P \in \PolL$, there
are unique $h^\phi_0, h^\phi_1,\dots,h^\phi_r$ such that 
\begin{equation}\label{eqn:h-vec}
    \phi_P(n) \ := \ \phi(nP) \ = \  h^\phi_0 \binom{n+r}{r} + h^\phi_1
    \binom{n+r-1}{r} + \cdots + h^\phi_r \binom{n}{r}
\end{equation}
for all $n \in \Z_{\ge 0}$.
Hence, every translation-invariant valuation $\phi$ comes with the notion of
an \mbox{$h^*$-vector} $h^\phi(P) := (h^\phi_0,h^\phi_1,\dots,h^\phi_d)$ with
$h^\phi_i = 0$ for $i > \dim
P$.  We call a valuation $\phi$ \Defn{combinatorially positive} if
$h^\phi_i(P) \ge 0$ and \Defn{combinatorially monotone} if  $h^\phi_i(P) \le
h^\phi_i(Q)$ whenever $P \subseteq Q$. The natural question that motivated the
research presented in this paper was
\begin{center}
    \it Which valuations are combinatorially positive/monotone?
\end{center}

The Euler characteristic shows that not every translation-invariant valuation
is combinatorially positive.  Beck, Robins, and Sam~\cite{BRS10} showed that
\emph{solid-angle} polynomials are combinatorially positive/monotone and they
gave a sufficient condition for combinatorial positivity/monotonicity of
general \emph{weight valuations}.  Unfortunately, this condition is not
correct; see the discussion after Corollary~\ref{cor:solid_NN}. We will revisit
the construction of weight valuations in Section~\ref{sec:vals} and show that
they are universal for $\LL = \Z^d$. Our main result is the
following simple complete characterization.

\begin{thm*}
    For a translation-invariant valuation $\phi : \PolL \rightarrow \R$, the
    following are equivalent:
    \begin{enumerate}[\rm (i)]
        \item $\phi$ is combinatorially monotone; 
        \item $\phi$ is combinatorially positive; 
        \item For every simplex $\Delta \in \PolL$
        \[
            \phi(\relint(\Delta)) \ := \ \sum_F (-1)^{\dim \Delta-\dim F}
            \phi(F) \ \ge \ 0,
        \]
        where the sum is over all faces $F \subseteq \Delta$.
        \end{enumerate}
\end{thm*}

The combinatorial positivity/monotonicity for the discrete volume
(Corollary~\ref{cor:Ehr_NN}) and solid angles (Corollary~\ref{cor:solid_NN})
are simple consequences and we show that Steiner polynomials are not
combinatorially positive (Example~\ref{ex:Steiner}). In
Section~\ref{sec:weak}, we investigate the relation of combinatorial
positivity/monotonicity to the more common notion of nonnegativity and
monotonicity of a valuation. In particular, we show that combinatorially
positive valuations are necessarily monotone and hence nonnegative. All
implications are strict.

Condition (iii) above is linear in $\phi$. Hence, the combinatorially positive
valuations constitute a pointed convex cone in the vector space of
translation-invariant valuations. In Section~\ref{sec:cones}, we investigate
the nested cones of combinatorially positive, monotone, and nonnegative
valuations. For $\LL = \R^d$, this gives a new characterization of the volume
as the unique, up to scaling, combinatorially positive valuation. For $\LL =
\Z^d$, these cones are more intricate. By results of Betke and
Kneser~\cite{BetkeKneser}, the vector space of valuations on $\Pol{\Z^d}$ that
are invariant under lattice transformations is of dimension $d+1$. We show
that the cone of lattice-invariant valuations that are combinatorially
positive is full-dimensional and simplicial. 

Hadwiger's characterization theorem~\cite{Hadwiger} states that the
coefficients of the Steiner polynomial give a basis for the continuous
rigid-motion invariant valuations on convex bodies that can be characterized
in terms of homogeneity, nonnegativity, and monotonicity, respectively. Betke
and Kneser~\cite{BetkeKneser} proved a discrete analog: a homogeneous basis
for the vector space of lattice-invariant valuations on $\Pol{\Z^d}$ is given
by the coefficients of the Ehrhart polynomial in the monomial basis.
Unfortunately, nonnegativity and monotonicity are genuinely lost. In
Section~\ref{sec:hadwiger} we prove a discrete characterization theorem: Up to
scaling there is a unique combinatorially-positive basis for lattice-invariant
valuations. We close with an explicit descriptions of the three cones of
combinatorially positive, monotone, and nonnegative lattice-invariant
valuations for $d=2$.

While Stanley's approach made use of the strong ties between Ehrhart
polynomials and commutative algebra, our main tool are \emph{half-open}
decompositions introduced by K\"oppe and \mbox{Verdolaage}~\cite{KV}.  We give
a general introduction to translation-invariant valuations in
Section~\ref{sec:vals} and we use half-open decompositions to give a simple
proof of McMullen's result~\eqref{eqn:h-vec} in Section~\ref{sec:halfopen}.
As a byproduct, we recover and extend the famous Ehrhart--Macdonald reciprocity
to multivariate Ehrhart polynomials in Section~\ref{sec:reciprocity}.\\

\textbf{Acknowledgements.} We would like to thank Christian Haase, Martin
Henk, and Monika Ludwig for stimulating discussions.  K.~Jochemko was
supported by a \emph{Hilda Geiringer Scholarship} at the Berlin Mathematical
School. R.~Sanyal was supported by European Research Council under the
European Union's Seventh Framework Programme (FP7/2007-2013)/ERC grant
agreement n$^\mathrm{o}$ 247029 and by the DFG-Collaborative Research Center,
TRR 109 ``Discretization in Geometry and Dynamics''.

\section{Translation-invariant valuations}\label{sec:vals}

Let $\LL \subset \R^d$ be a lattice (i.e.\ discrete subgroup) or a
finite-dimensional vector subspace over a subfield of $\R$.
Following~\cite{mcmullenEuler}, a convex polytope $P \subset \R^d$ with
vertices in $\LL$ is called a \Defn{$\boldsymbol{\LL}$-polytope} and we denote
all $\LL$-polytopes by $\PolL$.  A map $\phi : \PolL \rightarrow G$ into some
abelian group $G$ is a \Defn{valuation} if $\phi(\emptyset) = 0$ and $\phi$
satisfies the valuation property
\[
    \phi(P_1 \cup P_2) \ = \ \phi(P_1) + \phi(P_2) - \phi(P_1 \cap P_2)
\]
for all $P_1,P_2 \in \PolL$ with $P_1 \cup P_2, P_1 \cap P_2 \in \PolL$.  It
can be shown that valuations satisfy the more general
\Defn{inclusion-exclusion property}: For every collection $P_1,P_2,\dots,P_k
\in \PolL$ such that $P = P_1 \cup \cdots \cup P_k \in \PolL$ and $P_I :=
\bigcap_{i \in I} P_i \in \PolL$ for all $I\subseteq [k]$ 
\begin{equation}\label{eqn:IE}
    \phi(P) \ = \ \sum_{\emptyset \neq I \subseteq [k]} (-1)^{|I|-1}
    \phi(P_I).
\end{equation}
For $\LL$ a vector subspace this was first shown by Volland~\cite{volland};
for the case that $\LL$ is a lattice this is due to Betke (unpublished) in
the case of real-valued valuations and by McMullen~\cite{mcmullenIE} in
general.  A valuation $\varphi \colon \PolL \rightarrow G$ is
\Defn{translation-invariant} with respect to $\LL$ and called a
\Defn{$\LL$-valuation} if $\phi(t + P) = \phi(P)$ for all $P \in \PolL$ and $t
\in \LL$. We write $\Vals(\Lambda,G)$ 
for the family of $\LL$-valuations into $G$.

Many well-known valuations can be obtained as integrals over polytopes such
as the $d$-dimensional \Defn{volume} $\vol(P) = \int_P dx$. The volume is an
example of a \Defn{homogeneous} valuation, that is, $\vol(nP) = n^d \vol(P)$
for all $n \ge 0$.  An important valuation that can not be represented as an
integral is the \Defn{Euler characteristic} $\chi$ defined by $\chi(P) = 1$
for all non-empty polytopes $P$. The volume and the Euler characteristic are
$\LL$-valuations with respect to any $\LL$.  If $\LL$ is discrete, the
\Defn{discrete volume} $\dvol(P) := | P \cap \LL|$ is a
$\LL$-valuation.

We mention two particular techniques to manufacture new valuations from old
ones. If $\LL$ is a vector space over a subfield of $\R$, then $P \cap Q \in
\PolL$ whenever $P, Q \in \PolL$, i.e.~$\PolL$ is an intersectional family.
For a fixed valuation $\phi$ and a polytope $Q \in \PolL$, the map
\[
    \Int{\phi}{Q}(P) \ := \ \phi( P \cap Q )
\]
is a valuation.  Observe that $\Int{\phi}{Q}$ is not translation-invariant
unless $Q = \emptyset$.

The \Defn{Minkowski sum} of two $P,Q \in \PolL$ is the
$\LL$-polytope $P+Q = \{p + q : p \in P, q \in Q\}$. For a fixed
$\LL$-polytope $Q$ and valuation $\phi$, we define
\[
    \Mink{\phi}{Q}(P) \ := \ \phi(P + Q)
\]
for $P \in \PolL$. That this defines a valuation follows from the fact that
\begin{align*}
        (K_1\cup K_2)+K_3 &\ = \ (K_1 +K_3)\cup (K_2 +K_3)\\
        (K_1\cap K_2)+K_3 &\ = \ (K_1 +K_3)\cap (K_2 +K_3)
\end{align*}
for any convex bodies $K_1,K_2,K_3 \subset \R^d$;
cf.~\cite[Section~3.1]{Schneider}. Observe that $\Mink{\phi}{Q}$ is 
translation-invariant whenever $\phi$ is.

A result that we alluded to in the introduction regards the behavior of
$\LL$-valuations with respect to dilations. It was first shown for the
discrete volume by Ehrhart~\cite{ehrhart} and then for all
\mbox{$\LL$-valuations} by McMullen~\cite{mcmullenEuler}.

\begin{thm}\label{thm:phi_poly}
    Let $\phi : \PolL \rightarrow G$ be a $\LL$-valuation. For every
    $r$-dimensional $\LL$-polytope $P \subset \R^d$ there are unique 
    $h^\phi_0, h^\phi_1, \dots, h^\phi_r \in G$ such that 
    \[
        \phi_P(n) \ := \ \phi(nP) \ = \ 
     h^\phi_0 \binom{n+r}{r} + h^\phi_1 \binom{n+r-1}{r} + \cdots 
                 + h^\phi_r \binom{n}{r}.
     \]
\end{thm}

That is, $\phi_P(n)$ agrees with a polynomial for all $n \ge 0$.  We define the
\Defn{$\mathbf{h^*}$-vector} of $\phi$ and $P$ as the vector of coefficients
$h^\phi(P) := (h^\phi_0,\dots,h^\phi_d)$ with $h^\phi_i =  0$ for $i > \dim P$.  We will give
a simple proof of this result in Section~\ref{sec:halfopen} whose inner
workings we will need for our main results.

We define the \Defn{Steiner valuation} of a polytope $P \subset \R^d$ as
\[
    S(P) \ := \ \Mink{\vol}{B_d}(P) \ = \ \vol(P + B_d).
\]
Using Theorem~\ref{thm:phi_poly}, we obtain the \Defn{Steiner polynomial} 
\begin{equation}\label{eqn:Steinerpoly}
    S_P(n) \ := \ \vol(nP + B_d) \ = \ \sum_{i=0}^d \binom{d}{i}W_{d-i}(P)\,
    n^i.
\end{equation}
The coefficient $W_i(P)$, called the $i$-th \Defn{quermassintegral}, is a homogeneous valuation of degree $d-i$; see~\cite[Sect.~6.2]{gruber2007convex}.
The Steiner valuation is invariant under rigid motions and so are the
quermassintegrals.  Hadwiger's characterization theorem~\cite{Hadwiger} states
that for any real-valued valuation $\phi$ on convex bodies in $\R^d$ that is
continuous and invariant under rigid motions, there are unique
$\alpha_0,\dots,\alpha_d\in \R$ such that 
\[
    \phi \ = \ \alpha_0 W_0 + \cdots + \alpha_d W_d.
\]

Let $\LL$ be a lattice. A less well-known $\LL$-valuation is the solid-angle
valuation. The solid angle of a polytope $P$ at the origin is defined as
\[
    \omega(P) \ := \ \lim_{\eps \rightarrow 0} \frac{V(\eps B_d \cap P)}
    {V(\eps B_d)},
\]
where $B_d$ is the unit ball centered at the origin.  It is easy to see that
$\omega$ is a valuation.  The \Defn{solid-angle valuation} of $P \in \PolL$ is
defined as
\[
    A(P) \ := \ \sum_{p \in \LL} \omega(-p+P)
\]
By construction, this is a $\LL$-valuation and an example of a \Defn{simple}
valuation: $A(P) = 0$ whenever $\dim P < d$. 

Beck, Robins, and Sam~\cite{BRS10} considered a class of $\LL$-valuations
that generalize the idea underlying the solid-angle valuation. Slightly
rectifying the definitions in~\cite{BRS10}, a system of weights $\wSys =
(\nu_p)$ is a choice of a valuation $\nu_p : \Pol{\LL} \rightarrow G$ for
every lattice point $p \in \LL$ such that
\[
    N_\wSys(P) \ := \ \sum_{p \in \LL} \nu_p(P) 
\]
is defined for all $P \in \Pol{\LL}$.  Certainly a sufficient condition for
this is that $\nu_p$ has \emph{bounded support}, i.e.~$\nu_p(P) = 0$ whenever
$P \cap (R \cdot B_d-p) = \emptyset$ for some $R = R(\nu_p) > 0$.  We call
$N_\wSys$ a
\Defn{weight valuation}. If we choose $\nu_p(P) := \phi(-p + P)$ for some
fixed valuation $\phi$, then $N_\nu$ is a $\LL$-valuation. This generalizes
the solid-angle valuation for $\nu_p(P) = \omega(-p+P)$ as well as the
discrete volume for $\nu_p(P) = 1$ if and only if $p \in P$. For other
valuations it is in general not clear if they can be represented by weight
valuations.

\begin{example}[Euler characteristic]\label{ex:euler}
    Let $t \in \R^d$ be an irrational vector. For a non-empty lattice polytope
    $P \in \Pol{\Z^d}$ there is then always a unique vertex $v_t \in Q$ such
    that $\langle t, x \rangle \le \langle t, v_t \rangle$ for all $x \in Q$.
    Let $\nu_p$ be the function defined by $\nu_p(P) = 1$ if $v_t = p$ and
    zero otherwise.  In particular, $\nu_p(\emptyset) = 0$.  It is easy to
    check that this is a valuation and that $N_\nu$ is the Euler
    characteristic.
\end{example}

Before we ponder the general case, let us consider one more example.

\begin{example}[Volume]\label{ex:weight_vol}
    We write $C_d = [0,1]^d \subset \R^d$ for the standard cube and we define
    $\nu := \Int{\vol}{C_d}$. The induced weights are then 
    \[
        \nu_p(P) \ = \ \vol(P \cap (p + C_d))
    \]
    for $p \in \Z^d$. Since $\vol$ is a simple valuation, we get
    \[
        N_\wSys(P) \ = \ \vol \Bigl( \bigcup\{ P \cap (p + C_d) : p \in \Z^d\}
        \Bigr) = \vol(P).
    \]
\end{example}

The example already hints at the fact that general valuations on rational
polytopes can be expressed as weight valuations. The following result is
phrased in terms of the standard lattice $\LL = \Z^d$ but, of course, can be
adapted to any lattice $\LL$.

\begin{prop}\label{prop:weight_rat}
    Let $\phi : \Pol{\Q^d} \rightarrow G$ be a valuation on rational
    polytopes.  Then there is a system of weights $\wSys$ such that 
    $\phi\big|_{\Pol{\Z^d}} =
    N_\wSys$. 
\end{prop}
\begin{proof}
    Let $C_d = [0,1]^d$ be the standard cube and set $F_i := C_d \cap \{ x_i =
    0 \}$ for $i=1,\dots,d$. The set $H_d := C_d \setminus (F_1 \cup \cdots
    \cup F_d) = (0,1]^d$ is the \emph{half-open} standard cube. It is clear
    that $\{p + H_d\}_{p \in \Z^d}$ is a partition of $\R^d$.  Let us define the
    valuation
    \[
        \Int{\phi}{H_d} \ = \ \sum_{I \subseteq [d]} (-1)^{|I|}
        \Int{\phi}{F_I},
    \]
    where $F_I := \bigcap\{ F_i : i \in I\}$ and $F_\emptyset := C_d$. Then
    \[
        \sum_{p \in \Z^d} \phi( P \cap (p + H_d) ) \ = \
        \phi\Bigl( P \cap \biguplus\{p + H_d : p \in \Z^d\}\Bigr) = \phi(P),
     \]
     which proves the claim with $\nu _p(P)= \phi( P \cap (p + H_d) )$.
\end{proof}

Note that this result does not require $\phi$ to be invariant with respect to
translations. The main result of this section is a representation theorem for
$\Z^d$-valuations in terms of weight valuations.

\begin{thm}\label{thm:lat_weight_rep}
    Let $\phi : \Pol{\Z^d} \rightarrow G$ be a $\Z^d$-valuation taking values
    in divisible abelian group $G$. Then $\phi = N_\wSys$ for some system of
    weights $\wSys$.
\end{thm}

This result is a direct consequence of Proposition~\ref{prop:weight_rat} and
the following lemma which is of interest in its own right.

\begin{lem}\label{lem:ext_trans}
    Let $\phi : \Pol{\Z^d} \rightarrow G$ be a $\Z^d$-valuation taking values
    in a divisible abelian group.  Then there is a valuation $\bar \phi :
    \Pol{\Q^d} \rightarrow G$ that is invariant under translations by $\Z^d$
    and \mbox{$\bar \phi(P) = \phi(P)$} for all lattice polytopes $P \in \Pol{\Z^d}$.
\end{lem}
\begin{proof}
    Since $G$ is divisible, we can rewrite Theorem~\ref{thm:phi_poly} as
    \[
        \phi_P(n) \ = \ \phi_d(P) n^d + \cdots + \phi_0(P)
    \]
    for all $P\in \Pol{\Z^d}$. The coefficients $\phi_i(P)$ are 
    $\Z^d$-valuations homogeneous of degree $i$. It is sufficient to show that
    we can extend $\phi_i$ to rational polytopes.

    For $Q \in \Pol{\Q^d}$, let $\ell \in \Z_{>0}$ such that $\ell Q \in
    \Pol{\Z^d}$. We define
    \[
        \bar \phi_i (Q) \ := \ \tfrac{1}{\ell^i}\phi_i (\ell Q).
    \]
    To see that $\bar \phi_i$ is well-defined, observe that $\ell Q \in
    \Pol{\Z^d}$ if and only if $\ell=k\ell_0$ where $\ell_0$ is the least
    common multiple of the denominators of the vertex coordinates of $Q$ and
    $k \in \Z_{\ge 1}$.  Hence, by homogeneity
    \[
        \phi_i(\ell Q) = k^i \phi_i(\ell_0 Q).
    \]
    It remains to show that $\bar \phi_i$ satisfies the valuation property. Let
    $Q,Q'$ be rational polytopes such that $Q \cup Q' \in \Pol{\Q^d}$. Choose
    $\ell > 0$ such that $\ell Q, \ell Q', \ell (Q \cup Q')$ and $\ell(Q \cap
    Q')$ are lattice polytopes. Then
    \[
        \ell^i \bar \phi_i(Q \cup Q')  = 
        \phi_i(\ell(Q \cup Q'))
         =  \phi_i(\ell Q) + \phi_i(\ell Q') - \phi_i(\ell (Q \cap Q'))
         =  \ell^i \bar \phi_i(Q) + \ell^i \bar \phi_i(Q') - \ell^i \bar
\phi_i(Q \cap
        Q')
    \]
    which finishes the proof.
\end{proof}

Note that Lemma~\ref{lem:ext_trans} not necessarily yields the extension one
would expect: The discrete volume $\dvol$ clearly extends to rational
polytopes. However, the following example shows that this is not the extension
furnished by Lemma~\ref{lem:ext_trans}.

\begin{example}
    Consider the discrete volume $\dvol$ in dimension $d=1$. For lattice
    polytopes $P \subset \R$, the polynomial expansion is given by
    \[
        \dvol_P(n) \ = \ V(P)n + \chi(P),
    \]
    where $V$ is the $1$-dimensional volume. By Lemma~\ref{lem:ext_trans},
    there is an extension of $\dvol$ to rational segments and we compute
    \[
        \bar \dvol([0,\tfrac{1}{3}]) \ = \ \tfrac{1}{3}V(3[0,\tfrac{1}{3}]) +
        \chi(3[0,\tfrac{1}{3}]) \ = \ \tfrac{1}{3} + 1 \ \neq \ | Q \cap \Z |.
    \]
\end{example}

Since every abelian group $G$ can be embedded into a divisible group
$\overline{G}$, Theorem \ref{thm:lat_weight_rep} can be extended to abelian
groups if we allow the weights to take values in $\overline{G}$. However, the
assumption that $\phi$ is translation-invariant with respect to $\Z^d$ is
necessary for our proof. 

\begin{quest}
    Can Lemma~\ref{lem:ext_trans} be extended to general valuations $\phi :
    \Pol{\Z^d} \rightarrow G$?
\end{quest}

\section{Half-open decompositions and $h^{\ast}$-vectors}\label{sec:halfopen}

For a polytope $P \in \PolL$ and a valuation, we defined in the introduction 
\begin{equation}\label{eqn:phi_relint}
    \phi(\relint(P)) \ := \ \sum_{F} (-1)^{\dim P - \dim F} \phi(F),
\end{equation}
where the sum is over all faces $F$ of $P$. Using M\"obius inversion, this
definition is consistent with 
\begin{equation}\label{eqn:sumrelint}
    \phi(P) \ = \ \sum_{F} \phi(\relint F).
\end{equation}

In this section we will extend $\phi$ to half-open polytopes that allows us to
use half-open decompositions of polytopes for a proof of
Theorem~\ref{thm:phi_poly} that avoids inclusion-exclusion of any sort.

Let $P \subset \R^d$ be a full-dimensional polytope with facets
$F_1,\dots,F_m$. A point $q \in \R^d$ is \Defn{general} with respect to $P$ if
$q$ is not contained in any facet-defining hyperplane. The point $q$ is
\Defn{beneath} or \Defn{beyond} the facet $F_i$ if $q$ and $P$ are on the same
side or, respectively, on different sides of the facet hyperplane $\aff(F_i)$.
We write $I_q(P) \subset [m]$ for the set indexing the facets for which $q$ is
beyond. Since we assume $P$ to be full-dimensional, we always have $I_q(P)
\neq [m]$. A \Defn{half-open} polytope is a set of the form
\[
    \half{q}{P} \ := \ P \setminus \bigcup \{ F_i : i \in I_q(P) \}.
\]
We will write $\HO{P}$ for a half-open polytope $\half{q}{P}$ obtained
from $P$ with respect to some general point $q$.

Our interest in half-open polytopes stems from the following result of K\"oppe
and Verdolaage~\cite{KV} that is already implicit in the works of Stanley and
Ehrhart; see~\cite{stanley74}.  A \Defn{dissection} of a polytope $P$
is a presentation $P = P_1 \cup \cdots \cup P_k$, where each $P_i$ is a
polytope of dimension $\dim P$ and $\dim (P_i \cap P_j) < d$ for all $i \neq
j$.

\begin{lem}[{\cite[Thm.~3]{KV}}]\label{lem:ho-decomp}
    Let $P = P_1 \cup P_2 \cup \cdots \cup P_k$ be a dissection. If $q$ is a
    point that is general with respect to $P_i$ for all $i=1,\dots,k$, then
    \[
        \half{q}{P} \ = \ 
        \half{q}{P_1} \uplus
        \half{q}{P_2} \uplus
        \cdots \uplus
        \half{q}{P_k}.
    \]
\end{lem}

For sake of completeness we include a short proof of this result.

\begin{proof}
    We only need to show that for every $p \in \half{q}{P}$ there is a unique
    $P_i$ with $p \in \half{q}{P_i}$. There is a $P_i$ such that for every
    $\eps > 0$ sufficiently small, the point $p' := p + \eps(q-p)$ is in the
    interior and $p$ possibly in the boundary.
    In particular, the segment $[q,p]$ meets $P_i$ in the interior of $P_i$
    which shows that $p \in
    \half{q}{P_i}$. If $p \in P_j$ for some $j \neq i$, then there is a
    facet-hyperplane $H$ of $P_j$ through $p$ that separates $P_j$ from $p'$.
    This, however, shows that $q$ and $P_j$ are on different sides of $H$ and
    hence $p \not\in \half{q}{P_j}$.
\end{proof}

For a valuation $\phi$ we define 
\[
    \phi(\half{q}{P}) \ := \ \phi(P) - \sum_{\emptyset \neq J \subseteq I_q(P)}
    (-1)^{|J|} \phi( F_J ),
\]
where we set $F_J := \bigcap_{i \in J} F_i$. Lemma~\ref{lem:ho-decomp} now
implies the following.

\begin{cor}\label{cor:ho-val}
    Let $P = P_1 \cup \cdots \cup P_k$ be a dissection with $P_1,\dots,P_k \in
    \PolL$. If $\phi$ is a valuation on $\PolL$, then for a general $q \in
    \relint(P)$
    \[
        \phi(P) \ = \ 
        \phi(\half{q}{P_1}) + 
        \cdots + 
        \phi(\half{q}{P_k}).
    \]
\end{cor}

It is well-known (see for example~\cite{DLRS}) that every (lattice) polytope
$P$ can be dissected into (lattice) simplices. Thus,
Theorem~\ref{thm:phi_poly} follows from Corollary~\ref{cor:ho-val} and the
following proposition.

\begin{prop}\label{prop:ho-poly}
    Let $\HO{S}$ be a full-dimensional, half-open $\LL$-simplex and $\phi$ a
    $\LL$-valuation.  Then the function $\phi_{\HO{S}}(n) = \phi(n\HO{S})$ is
    a polynomial in $n$ of degree at most $d$. 
\end{prop}
\begin{proof}
    Let $S$ be the $\LL$-simplex such that $\HO{S} = \half{q}{S}$ for some
    general $q$ and set $I = I_q(S)$. Now, $S$ has vertices $v_1,\dots,v_{d+1}$ and facets
    $F_1,\dots,F_{d+1}$ labeled in such a way that $v_i \not\in F_i$ for $i
    \in [d+1]$.  An intrinsic description of $\HO{S}$ is given by
    \[
        \HO{S} \ = \ \left\{ \sum_i \lambda_i v_i : \sum_i \lambda_i = 1,
        \lambda_i \ge 0 \text{ for } i \not\in I, \lambda_i > 0 \text{ for } i
        \in I\right\}.
    \]
    Define $\bar v_i = (v_i,1) \in \R^{d+1}$ and consider the half-open
    polyhedral cone
    \[
        C \ := \ \bigl\{ \mu_1 \bar v_1 + \cdots + \mu_{d+1} \bar v_{d+1} :
        \mu_1,\dots,\mu_{d+1} \ge 0, \mu_i > 0 \text{ for } i \in I \bigr\}.
    \]
    For $n \ge 0$, the hyperplane $H_n = \{ x \in \R^{d+1} : x_{d+1} = n \}$
    can be naturally identified with $\R^d$ such that $H_n \cap C =
    n\,\HO{S}$, where $0\,\HO{S}:=\emptyset$ unless $I=\emptyset$.  Define the 
    (half-open) parallelepiped
    \[
        \Pi \ := \ \bigl\{ \mu_1 \bar v_1 + \cdots + \mu_{d+1} \bar v_{d+1} :
        0 \le \mu_i < 1 \text{ for } i \not\in I, 0 < \mu_i \le 1 \text{ for }
        i \in I \bigr\}.
    \]
    Then for every $p \in C$ there are unique $\mu_i \in \Z_{\ge0}$ and $r \in 
    \Pi$ such that $p = \sum_i \mu_i \bar
    v_i + r$.  Let us write 
    \begin{equation}\label{eqn:hypersimplex} 
        \Pi_i \ := \ \Pi \cap H_i \quad \text{ for } 0 \le i \le d. 
    \end{equation}        
    In general, the $\Pi_j$ are \emph{not} half-open polytopes but
    \emph{partly-open}: they are $\LL$-polytopes with certain relatively open
    faces removed.  It follows that
    \[
        n\, \HO{S} \ = \ C \cap H_n \ = \ \biguplus _{k,r \ge 0 \atop 
        k+r=n}\{ \bar v_{i_1} + 
        \cdots +
        \bar v_{i_k} + \Pi_r : 1 \le i_1 \le \cdots \le i_k \le d+1\}.
    \]
    This is a partition of $n\, \HO{S}$ into partly-open polytopes. Using the
    translation-invariance of $\phi$ yields
\begin{equation}\label{eqn:halfopenpoly}
        \phi_{\HO{S}}(n) \ = \ 
        \phi(\Pi_0) \binom{n+d}{d} + 
        \phi(\Pi_1) \binom{n+d-1}{d} + 
        \cdots + 
        \phi(\Pi_d) \binom{n}{d},
        \end{equation}
    where we used~\eqref{eqn:sumrelint} to compute $\phi(\Pi_j)$.
\end{proof}

A notion developed in the proof that will be of importance later is the
following. For a (half-open) simplex $S$, we define the $j$-th \Defn{(partly
open) hypersimplex} $\Pi_j(S)$ through~\eqref{eqn:hypersimplex}.
Proposition~\ref{prop:ho-poly} prompts the definition of an $h^*$-vector for
half-open polytopes. The proof of Proposition~\ref{prop:ho-poly} then yields
\begin{cor}\label{cor:h-hyper}
    If $\HO{S} \subset \R^d$ is a half-open $\LL$-simplex and $\phi$ a
    $\LL$-valuation, then
    \[
        h^\phi_j(\HO{S}) \ = \ \phi(\Pi_j(\HO{S}))
    \]
    for all $0 \le j \le d$.
\end{cor}

The following is an immediate consequence of Corollary~\ref{cor:ho-val} and
Proposition~\ref{prop:ho-poly}.

\begin{cor}\label{cor:h-add}
    Let $P \in \PolL$ be a polytope and $\phi$ a $\LL$-valuation.  Let $P =
    P_1 \cup \cdots \cup P_k$ be a dissection into $\LL$-simplices and $q \in
    \relint(P)$ a point general with respect to $P_i$ for all $i=1,\dots,k$.
    Then
    \[
        h^\phi(P) \ = \ 
        h^\phi(\half{q}{P_1}) + \cdots +
        h^\phi(\half{q}{P_k}).
    \]
\end{cor}

\subsection{Combinatorial positivity and monotonicity}\label{ssec:h-non}

We now assume that $G$ is an abelian group together with a partial order
$\preceq$ compatible with the group structure, that is, $(G,\preceq)$ is a
poset such that for all $a,b,c \in G$ 
\[
    a \ \preceq \ b \quad \Longrightarrow \quad
    a+c \ \preceq \ b+c .
\]
    
A $\LL$-valuation $\phi :  \PolL \rightarrow G$ is called
\Defn{combinatorially positive} or \Defn{$h^\ast$-nonnegative} if
\[
    h^\phi_i (P) \ \succeq \  0
\]
for all $P\in \PolL$ and $0 \le i \le d$ and \Defn{combinatorially monotone}
or \Defn{$h^*$-monotone} if
\[
h^\phi_i (P) \ \preceq \ h^\phi_i (Q)
\]
for $P \subseteq Q \in \PolL$ and $0 \le i \le d$.  Our main theorem from the
introduction is a special case of the following.

\begin{thm}\label{thm:h-nonneg}
    For a $\LL$-valuation $\phi : \PolL \rightarrow G$ into a partially
    ordered abelian group $G$, the following are equivalent:
    \begin{enumerate}[\rm (i)]
        \item $\phi$ is combinatorially monotone; 
        \item $\phi$ is combinatorially positive; 
        \item $\phi(\relint(\Delta)) \succeq 0$ for every $\LL$-simplex $\Delta$.
    \end{enumerate}
\end{thm}

\begin{proof}
    The implication (i) $\Rightarrow$ (ii) simply follows from the fact that 
    $\emptyset$ is trivially a $\LL$-polytope. Hence, $h^\phi_i(P) \succeq
    h^\phi_i(\emptyset) = 0$ for every $P \in \PolL$ and all $i$.

    For (ii) $\Rightarrow$ (iii), let $\Delta$ be a $\LL$-simplex of dimension
    $r$. Note that the $(r-1)$-th partly-open hypersimplex $\Pi_{r-1}$ of
    $\Delta$ is a translate of $\relint(-\Delta)$. Combinatorial positivity
    implies that \mbox{$0 \preceq h^\phi_{r-1}(-\Delta) = \phi(\Pi_{r-1}(-\Delta)) =
    \phi(\relint(\Delta))$}.
    
    (iii) $\Rightarrow$ (i): Let $P\subseteq Q$ be two $\LL$-polytopes.  If $r
    = \dim P = \dim Q$, let $Q = T_1 \cup T_2 \cup \cdots \cup T_N$ be a
    dissection of $Q$ into $r$-dimensional $\LL$-simplices such that $P =
    T_{M+1} \cup T_{M+2} \cup \cdots \cup T_N$ for some $M < N$.  Such a
    dissection can be constructed using, for example, the
    \emph{Beneath-Beyond} algorithm~\cite[Section~8.4]{edelsbrunner}.
    For a point $q \in \relint P$ general with respect to all $T_i$, it
    follows from Corollary~\ref{cor:h-add} that
    \[
        h^\phi_i(Q) - h^\phi_i(P) \ = \ 
        h^\phi_i(\half{q}{T_1}) + 
        \cdots + 
        h^\phi_i(\half{q}{T_M}).
    \]
    Hence, it is sufficient to show 
    \[
        h^\phi_i(\HO{S}) \succeq 0
    \]
    for any \emph{half-open} $\LL$-simplex $\HO{S}$.  For $0 \le i \le \dim
    \HO{S}$, let $\Pi_i = \Pi_i(\HO{S})$ be the corresponding \mbox{$i$-th}
    hypersimplex and let $\overline{\Pi}_i$ be its closure.  Pick a
    triangulation $\Triang$ of $\overline{\Pi}_i$ into $\LL$-simplices.  Then
    \mbox{$\Triang^\prime = \{ \sigma \in \Triang : \relint(\sigma) \subset
    \Pi_i\}$}
    is a triangulation of the partly-open hypersimplex. From
    Corollary~\ref{cor:h-hyper} and inclusion-exclusion, we obtain
    \[
        h^\phi_i(\HO{S}) \ = \ \phi(\Pi_i) \ = \ \sum_{\sigma \in
        \Triang^\prime} \phi(\relint{\sigma}) \ \succeq \ 0,
    \]
    which completes the proof for the case $\dim P = \dim Q$. 

    Let $r := \dim Q - \dim P > 0$. Set $P^0 := P$ and 
    $P^i := \conv(P^{i-1} \cup q_i)$ for $i = 1,\dots,r-1$,
    where $q_i \in (Q \cap \LL)
    \setminus \aff(P^{i-1})$.  This yields a chain of $\LL$-polytopes
    \[
        P \ = \ P^0 \ \subset \ P^1 \ \subset \ \cdots \ \subset \ P^r \
        \subseteq \ Q
    \]
    with $\dim P^i = \dim P^{i-1} + 1$ for $1 \le i \le r$.  So, it remains to
    prove that $h^\phi(P) \preceq h^\phi(Q)$ when $Q$ is a pyramid with base $P$
    and apex $a$.  Let $P = P_1 \cup \cdots \cup P_k$ be a dissection of $P$
    into $\LL$-simplices. This induces a dissection of $Q$ with pieces $Q_i =
    \conv(P_i \cup a)$. A point $q \in \relint Q$ general with respect to
    all $Q_i$, gives half-open simplices $\HO{Q}_i = \half{q}{Q_i}$ with
    half-open facets $\HO{P}_i = \HO{Q}_i \cap P_i$. For $0 \le j \le d$, it
    is easy to see that $\Pi_j(\HO{P}_i) \subseteq \Pi_j(\HO{Q}_i)$ is a
(partly open) face.
    For fixed $j$, we compute from a triangulation $\Triang$ of
    $\Pi_j(\HO{Q}_i)$
    \[
        h^\phi_j(\HO{Q}_i) - h^\phi_j(\HO{P}_i) \ = \ \sum \{
        \phi(\relint(\sigma)) : \sigma \in \Triang, \relint(\sigma) \not
        \subseteq \Pi_j(\HO{P}_i) \} \ \succeq \ 0
    \]
    and hence
    \[
        h^\phi_j(Q) - h^\phi_j(P) \ = \ \sum_i h^\phi_j(\HO{Q}_i) -
        h^\phi_j(\HO{P}_i) \ \succeq 0. \qedhere
    \]
\end{proof}

As a direct consequence we recover Stanley's results regarding the
$h^*$-vector for the discrete volume.

\begin{cor}\label{cor:Ehr_NN}
    Let $\LL$ be a lattice. The discrete volume $\dvol(P) = |P\cap \Z^d|$ is a
    $h^*$-nonnegative and $h^*$-monotone valuation.
\end{cor}
\begin{proof}
    By Theorem~\ref{thm:h-nonneg}, it suffices to prove that
    $\dvol(\relint(P)) \ge 0$ for all polytopes $P \in \Pol{\Z^d}$.
    From the definition of $\dvol(\relint(P))$ it follows 
    that $\dvol(\relint(P)) = |\relint(P) \cap \Z^d| \ge
    0$. 
\end{proof}

Another simple application gives the following.
\begin{cor}\label{cor:simple_NN}
    A simple $\LL$-valuation $\phi : \PolL \rightarrow G$ is combinatorially
    positive
    if and only if $\phi(P) \succeq 0$ for all $P \in \PolL$.
\end{cor}
\begin{proof}
    For a simple valuation, we observe that
    \[
        \phi(\relint (P)) \ = \ \sum_F (-1)^{\dim (P)-\dim (F)} \phi(F) \
        = \ \phi (P).
    \]
    Theorem~\ref{thm:h-nonneg} yields the claim.
\end{proof}

Since the solid-angle valuation is simple, this implies the main results of
Beck, Robins, and Sam~\cite{BRS10}.

\begin{cor}\label{cor:solid_NN}
    The solid-angle valuation $A(P)$ is $h^*$-nonnegative and $h^*$-monotone.
\end{cor}

Beck, Robins, and Sam also give a sufficient condition for the
$h^*$-nonnegativity/-monotonicity of general weight valuations. Theorems~3
and~4 of~\cite{BRS10} state that $N_\wSys$ is \mbox{$h^*$-nonnegative} and
\mbox{$h^*$-monotone} if and only if $\nu_p(P) \ge 0$ for all $P \in \Pol{\Z^d}$ and
all $p \in \Z^d$. Unfortunately, this condition is not correct as
Example~\ref{ex:euler} shows.

The Steiner valuation $S$ also turns out not to be combinatorially
positive/monotone.
\begin{example}\label{ex:Steiner}
    Let $P=[0,\alpha e_1] \subset \R^d$ be a segment of length $\alpha >0$ in
    dimension $d > 1$. Then
    \[
        S(\relint (P)) \ = \ \vol( P + B_d ) - \vol(0 + B_d) - \vol(\alpha e_1
        + B_d)
        \ = \ 
        \alpha {\vol}_{d-1}(B_{d-1})-{\vol}_d(B_d) \ < \ 0
    \]
    for $\alpha$ sufficiently small.

\end{example}

\section{Reciprocity and a multivariate Ehrhart--Macdonald
Theorem}\label{sec:reciprocity}

A fascinating result in Ehrhart theory and an important tool in geometric and
enumerative combinatorics is the reciprocity theorem of Ehrhart and Macdonald.

\begin{thm}\label{thm:ehrhart_macdonald}
    Let $P \subset \R^d$ be a lattice polytope and $\dvol_P(n)$ its Ehrhart
    polynomial. Then
    \[
        (-1)^{\dim P} \dvol(-n) \ = \ \dvol(\relint(nP)) = |\relint(nP) \cap
        \Z^d|.
    \]
\end{thm}

McMullen~\cite{mcmullenEuler} generalized this result to all
$\Lambda$-valuations as follows.

\begin{thm}\label{thm:reciprocity}
    Let $\phi : \PolL \rightarrow G$ be a $\LL$-valuation and $P \in \PolL$.
    Then
    \[
        (-1)^{\dim P}\phi_P(-n) \ = \ \phi(\relint(-nP)).
    \]        
\end{thm}

In this section we succumb to the temptation to give a simple proof of
Theorem~\ref{thm:reciprocity} using the machinery of half-open decompositions
developed in Section~\ref{sec:halfopen}. As a corollary we obtain McMullen's
multivariate version of Theorem~\ref{thm:phi_poly} for Minkowski sums
$\phi(n_1 P_1 + \cdots + n_k P_k)$ and, from the perspective of weight
valuations, we give a multivariate Ehrhart--Macdonald reciprocity
(Theorem~\ref{thm:multireciprocity}). This section is not necessary for the remainder of the
paper and can, if necessary, be skipped.

We start with a generalization of Lemma~\ref{lem:ho-decomp}. Let $P \subset
\R^d$ be a full-dimensional polytope with facets $F_1,\dots,F_m$. 
For a general point $q \in \R^d$, we defined $I_q(P) = \{ i \in [m] : q \text{
beyond } F_i\}$ which led us to the definition of half-open polytopes. We now
define 
\[
    \ihalf{q}{P} \ := \ P \setminus \bigcup \{ F_i : i \not\in I_q(P) \} \ = \
    P \setminus \overline{\partial \half{q}{P}}.
\]
In a more general setting the relation between $\half{q}{P}$ and
$\ihalf{q}{P}$ was studied in~\cite{AdiprasitoS}.

\begin{lem}\label{lem:ho-decomp-complement}
    Let $P = P_1 \cup P_2 \cup \cdots \cup P_k$ be a dissection and $q$
    general with respect to all $P_i$. Then
    \[
    \ihalf{q}{P} \ = \ 
    \ihalf{q}{P_1} \uplus
    \ihalf{q}{P_2} \uplus
    \cdots \cup
    \ihalf{q}{P_k}.
    \]
\end{lem}
\begin{proof}
    For a polytope $P \subset \R^d$, define the homogenization $\Hom{P} :=
    \{ (x,t) : t \ge 0, x \in t P \}$. This is a polyhedral cone and $P$ can be 
    identified with $\rho(\Hom{P}) := \{ (x,1) \in \Hom{P}\}$. Let $\Hom{q} 
    =\binom{q}{1}$. Then $\ihalf{q}{P_i} = \rho(\half{-\Hom{q}}{\Hom{P_i}})$.
    Applying Lemma~\ref{lem:ho-decomp} with $-\Hom{q}$ to
    \[
        \Hom{P} \ = \  
        \Hom{P}_1 \cup
        \cdots \cup
        \Hom{P}_k
    \]
    then proves the claim.
\end{proof}

The following reciprocity is a simple extension of Stanley's result for
reciprocal domains; see~\cite{stanley74}.  Observe that for $q \in
\relint(P)$, we get $\ihalf{q}{P} = \relint(P)$ and hence the following
theorem subsumes Theorem~\ref{thm:reciprocity}.

\begin{thm}\label{thm:halfopenrec}
    Let $P$ be a $\LL$-polytope, and $\phi$ be a $\LL$-valuation. Then
    \[
    (-1)^{\dim P}\phi _{\half{q}{P}}(-n) \ = \ \phi (-n\ihalf{q}{P}).
    \]
\end{thm}

\begin{proof}
    Since $\phi_{\half{q}{P}}(n) = \phi_{n\half{q}{P}}(1)$, we only have to
    prove that $(-1)^{\dim P}\phi_{\half{q}{P}}(-1) = \phi(-\ihalf{q}{P})$.
    Let us first assume that $P$ is a simplex of dimension $d$. With the
    notation taken from the proof of Proposition~\ref{prop:ho-poly} and
    equation \eqref{eqn:halfopenpoly} we obtain
    \begin{equation}\label{eqn:halfopenrec}
        (-1)^{\dim P}\phi_{\half{q}{P}}(-n) \ = \ 
        \phi(\Pi_0) \binom{n-1}{d} + 
        \phi(\Pi_1) \binom{n}{d} + 
        \cdots + 
        \phi(\Pi_d) \binom{n+d-1}{d},
    \end{equation}
    where $\Pi_i = \Pi_i(\half{q}{P})$ and we used the identity $(-1)^b
    \binom{-a + b}{b} = \binom{a-1}{b}$.  Thus,
    \[
        (-1)^{\dim P}\phi_{\half{q}{P}}(-1)=\phi (\Pi_d)=\phi(-\ihalf{q}{P}),
    \]
    since $\Pi _d$ is a translate of $-\ihalf{q}{P}$. Now, let $P$ be an
    arbitrary $\LL$-polytope, and let $P=T_1\cup \cdots \cup T_k$ be a
    dissection into $\LL$-simplices. Then
    \begin{eqnarray*}
        (-1)^{\dim P}\phi_{\half{q}{P}}(-1) & = & 
        (-1)^{\dim P}\left( \phi_{\half{q}{T_1}}(-n)+ \cdots +
            \phi_{\half{q}{T_k}}(-n) \right)\\
            & = & \phi(-\ihalf{q}{T_1}) + \cdots + \phi(-\ihalf{q}{T_k}) \\
            & = &\phi(-\ihalf{q}{P})
    \end{eqnarray*}
    by Lemma \ref{lem:ho-decomp-complement}.
\end{proof}

\begin{cor}\label{cor:reciprocity_weights}
    Let $N_{\wSys}$ be a translation-invariant weight valuation and $P$ be a 
    lattice polytope. Then, also $P\mapsto (-1)^{\dim P}N_{\wSys} 
    (-\relint P)$ 
    is a weight valuation, and
    \[
    (-1)^{\dim P}(N_{\wSys})_P (-n)=\sum_{p\in \Z^d}\nu_p(\relint(-nP)).
    \]
\end{cor}

\subsection{Multivariate Ehrhart--Macdonald reciprocity}
\label{ssec:mult_reciprocity}

A multivariate version of Theorem \ref{thm:phi_poly} was given by 
Bernstein~\cite{Bernstein} for the discrete volume and by
McMullen~\cite{mcmullenEuler} for general $\LL$-valuations.

\begin{thm}[{\cite[Theorem~6]{mcmullenEuler}}]\label{Thm:multipolynomial}
    Let $\phi : \PolL \rightarrow G$ be a $\LL$-valuation and let
    $P_1,\ldots ,P_k \in \PolL$. Then the function
    \[
        \phi_{P_1,\dots,P_k}(n_1,\dots,n_k) \ = \ \phi( n_1 P_1 + \cdots + n_k
        P_k)
    \]
    agrees with a polynomial of total degree at most 
    $\dim P_1+\cdots + P_k$ for all $n_1,\dots,n_k \ge 0$.
\end{thm}
\begin{proof}
    For $k=1$, this is just Theorem~\ref{thm:phi_poly}. For $k > 1$, consider
    for fixed $P_k$ the $\LL$-valuation $\Mink{\phi}{P_k}$. By induction, 
    $\phi_{P_1,\dots,P_{k-1}}(P_k;n_1,\dots,n_{k-1}) :=
    \Mink{\phi}{P_k}(n_1P_1 + \cdots +
    n_{k-1}P_{k-1})$ is a polynomial in $n_1,\dots,n_{k-1}$. In particular, the map 
    \[
        P_k \ \mapsto \ \phi_{P_1,\dots,P_{k-1}}(P_k) \ := \
        \phi_{P_1,\dots,P_{k-1}}(P_k;n_1,\dots,n_{k-1}) \in G[n_1,\dots,n_{k-1}]
    \]
    is a $\LL$-valuation. Hence, again by Theorem~\ref{thm:phi_poly}, 
    \[
    (\phi_{P_1,\dots,P_{k-1}})_{P_k}(n_k) \ = \ \phi(n_1P_1 + \cdots + n_{k-1}P_{k-1} + n_k P_k) \
        \in \ G[n_1,\dots,n_{k-1}][n_k]
    \]
    is a multivariate polynomial. The total degree of
    $\phi_{P_1,\dots,P_k}(n_1,\dots,n_k)$
    is equal to the degree of $\phi_{P_1,\dots,P_k}(n,n,\dots,n) =
    \phi(n(P_1+\dots+P_k))$ in $n$ which, by Theorem~\ref{thm:phi_poly}, is 
    \mbox{$\le \dim(P_1 + \cdots + P_k)$}.

\end{proof}

Specializing Theorem \ref{Thm:multipolynomial} to the discrete volume yields
that for lattice polytopes $P_1,\dots,P_k \subset \R^d$ 
\[
    \dvol_{P_1,\dots,P_k}(n_1,\dots,n_k) \ = \ | (n_1P_1 + \cdots + n_kP_k) \cap
    \Z^d |
\]
agrees with a polynomial for all $n_1,\dots,n_k \ge 0$. Using
Ehrhart--Macdonald reciprocity (Theorem~\ref{thm:ehrhart_macdonald}), we can interpret $(-1)^r
\dvol_{P_1,\dots,P_k}(-n_1,\dots,-n_k)$ for $n_1,\dots,n_k \ge 0$ as the number
of lattice points in the relative interior of $P = n_1 P_1 + \cdots + n_k P_k$
where $r = \dim P$. This raises the natural question if there is a
combinatorial interpretation for the evaluation
\begin{equation}\label{eqn:gen_rec}
    \dvol_{P_1,\dots,P_k}(-n_1,\dots,-n_l,n_{l+1},\dots,n_k)
\end{equation}
for $n_1,\dots,n_k \ge 0$ and $1 < l < k$. The following example shows that
there cannot be a straightforward generalization of
Theorem~\ref{thm:ehrhart_macdonald}.

\begin{example}
    Let $P=[0,1]^2$ and $Q=[(0,0),(1,1)]$. Then
    \[
        \dvol _{P,Q}(n,m)=(n+1)^2+2nm+m.
    \]
    Therefore
    \begin{align*}
      \dvol _{P,Q}(-n,m) & \ < \ 0 \qquad \text{ for } 0 < n \ll m ,\\
      \dvol _{P,Q}(-n,m) & \ > \ 0  \qquad \text{ for } 0 < m \ll n .\\
    \end{align*}
\end{example}

However, from the perspective of weight valuations, we can give an
interpretation of~\eqref{eqn:gen_rec} in terms of the topology of certain
polyhedral complexes. We first note that for~\eqref{eqn:gen_rec}
\[
    \dvol_{P_1,\dots,P_k}(-n_1,\dots,-n_l,n_{l+1},\dots,n_k) \ = \
    \dvol_{P,Q}(-1,1) \ = \ \Mink{\dvol}{Q}_P(-1)
\]
where $P := n_1P_1 + \cdots +  n_l P_l$ and $Q = n_{l+1}P_{l+1} + \cdots +
n_kP_k$. Hence, it is sufficient to find an interpretation for
$\dvol_{P,Q}(-1,1)$ for general lattice polytopes $P, Q$.

For two polytopes $P,Q \subset \R^d$, the \Defn{$\boldsymbol{Q}$-complement} is 
the polyhedral complex
\[
    \Qcomp{Q}{P} \ := \ 
    \{F\subseteq P \text{ face} :  F \cap Q = \emptyset\}.
\]
Recall that the \Defn{reduced Euler characteristic} of a polyhedral complex
$\mathcal{K}$ is defined as $\Euler(\mathcal{K}) := \sum \{ (-1)^{\dim F} : F
\in \mathcal{K}\}$. 
Here is our generalization of Ehrhart--Macdonald
reciprocity to Minkowski sums of lattice polytopes.

\begin{thm}\label{thm:multireciprocity}
    Let $P,Q \subset \R^d$ be non-empty lattice polytopes. Then 
    \[
        P \mapsto \Euler(\Qcomp{Q}{P}) 
    \]
    defines a valuation on $\Pol{\Z^d}$ and 
    \begin{equation*}\label{eq:weights}
        \dvol_{P,Q}(-1,1) \ = \ - \sum _{p\in \Z^d} \Euler(\Qcomp{Q}{P + p}).
    \end{equation*}
\end{thm}
    
\begin{proof}
    Consider $\phi := \Int{\chi}{(-Q)}$ and define a system of
    weights $\wSys$ by \mbox{$\nu_p(P) := \phi(-p + P)$}. We have $\nu_p(P) = 1$ if
    and only if $(-p + P) \cap (-Q) \neq \emptyset$ if and only if $p \in P +
    Q$. Hence,
    \[
        \Mink{\dvol}{Q}(P) \ = \ \sum_{p \in \Z^d} \nu_p(P) \ = \ N_\wSys(P).
    \]
    By Corollary~\ref{cor:reciprocity_weights}, we obtain 
    \begin{eqnarray*}
        \Mink{\dvol}{Q}_P(-1)
        & = &\sum_{p\in \Z^d} (-1)^{\dim(P)} \Int{\chi}{(-Q)}(-(p + \relint
                P))\\
        & = & \sum_{p\in \Z^d}\sum \{ (-1)^{\dim F} : F\subseteq P \text{
        face}, (F+p) \cap Q\neq  \emptyset\}\\
        & = & - \sum_{p\in \Z^d}\Euler(\Qcomp{Q}{P+p})
    \end{eqnarray*}
    where the last equation follows from the fact that the complex of faces of
    $P$ has reduced Euler characteristic $=0$.
\end{proof}

For $Q = \{0\}$, we recover  Ehrhart--Macdonald reciprocity: For $p \in
\Z^d$, set 
\[
    \mathcal{C}_p := \Qcomp{Q}{-p+P} = \{ F \subseteq P \text{ face} : p
\not\in F\}.
\] 
For $p \in \relint(P)$, $\mathcal{C}_p$ is a sphere of dimension
$\dim P -1$. For $p \not\in P$ and $p \in \partial P$, the complex
$\mathcal{C}_p$ is a ball and hence $\Euler(\mathcal{C}_p) = 0$. Hence,
Theorem~\ref{thm:multireciprocity} yields
\[
    \dvol_P(-1) \ = \ \sum_{p \in \relint(P) \cap \Z^d} (-1)^{\dim P} \ = \
    (-1)^{\dim P} \dvol(\relint P).
\]

One could hope that the $Q$-complements are combinatorially well-behaved
(e.g. shellable, Cohen-Macaulay, Gorenstein, etc.), but it turns out that
$Q$-complements are universal.

\begin{prop}
    Let $\mathcal{C}$ be a simplicial complex. Then there are lattice
    polytopes $P$ and $Q$ such that 
    \[
        \mathcal{C} \ \cong \ \Qcomp{Q}{P}.
    \]
\end{prop}
\begin{proof}
Let $\mathcal{C}$ be a simplicial complex on the vertex set $[m]$.  Let $P =
\conv(e_1,\dots,e_m)\subset \R^m$ be a lattice $(m-1)$-simplex. For $I\subseteq 
[m]$
let
\[
    w_I \ := \ \frac{1}{|I|}\sum _{i\in I} e_i
\]
be the barycenter of the face $F_I := \conv(e_i \colon i\in I) \subseteq P$.
Let $Q=\conv(w_I \colon I \not \in \mathcal{C})$. Then $F_I \cap Q = \emptyset$
if and only if $I \not\in \mathcal{C}$. Hence, $\Qcomp{Q}{P}$ is a geometric
realization of $\mathcal{C}$. Observing that $m!Q \subseteq m!P$ are lattice
polytopes finishes the proof.
\end{proof}

In particular, the weights appearing in Theorem \ref{thm:multireciprocity} can
be arbitrary. This, however, does not exclude the possibility that there are
combinatorial interpretations of $\dvol_{P,Q}(m,n)$ for certain regimes
$\mathcal{R} \subset \Z^2$ and it would certainly be interesting to find such
interpretations.

\section{Weak $h^*$-nonnegativity, monotonicity, and nonnegativity}
\label{sec:weak}

The Euler characteristic is a simple example of $\LL$-valuation that is not
combinatorially positive. Indeed, for a $r$-polytope $P \neq \emptyset$ we have
\[
    h^\chi_i(P) \ = \ (-1)^i \binom{r}{i}.
\]
In this section we consider a weaker notion than $h^*$-nonnegativity that
clarifies the relation of combinatorial positivity/monotonicity to the usual
nonnegativity and monotonicity of valuations. A $\LL$-valuation $\phi \in
\Vals(\LL,G)$ is \Defn{weakly combinatorially monotone} or \Defn{weakly
$h^*$-monotone} if $\phi(\{0\})\succeq 0$ and 
\[
    h^\phi_i (P) \ \preceq \ h^\phi_i (Q) 
\]
for all $\Lambda$-polytopes $P \subseteq Q$ such that $\dim (P)=\dim (Q)$.
Clearly, every combinatorially monotone valuation is also weakly
combinatorially monotone. Moreover, the Euler characteristic is weakly
$h^*$-monotone which also shows that weakly $h^*$-monotone does not imply
$h^*$-monotone. The main result of this section exactly characterizes the
weakly $h^*$-monotone valuations.

\begin{thm}\label{thm:h-weaknonneg}
    For a $\LL$-valuation $\phi : \PolL \rightarrow G$ into a partially
    ordered abelian group $G$, the following are equivalent:
    \begin{enumerate}[\rm (i)]
        \item $\phi$ is weakly $h^*$-monotone;
        \item $\phi(\relint(\Delta)) + \phi(\relint(F)) \succeq 0$ for every 
        $\LL$-simplex $\Delta$ and every facet $F$ of $\Delta$;
        \item $\phi (\HO{S}) \succeq 0$ for every half-open $\LL$-simplex
            $\HO{S}$.
    \end{enumerate}
\end{thm}

\begin{proof}
    (i) $\Rightarrow$ (ii): Let $\Delta = \conv(v_0,\dots,v_r)$ be a
    $\LL$-simplex of dimension $r$. We can assume that $v_0=0$. If $r = 0$,
    then $\phi(\relint\Delta)\geq 0$ by definition. For $r > 0$, the truncated
    pyramid $T = \overline{2\Delta \setminus \Delta}$ is contained in
    $2\Delta$ and is of dimension $r$. Using that $\phi$ is weakly
    $h^*$-monotone, we obtain
    \[
        0 \ \preceq \ h^{\phi}_r(-2\Delta) - h^{\phi}_r(-T) 
        \ = \ \phi(\relint(2\Delta)) - \phi(\relint(T))
        \ = \ \phi(\relint(\Delta)) + \phi(\relint(F)),
    \]
    where $F$ denotes the facet opposite to $v_0 = 0$.

    (ii) $\Rightarrow$ (iii): Let $\HO{S}$ be a half-open simplex of dimension
    $r$ and let $f = f(\HO{S})$ be the number of facets present in $\HO{S}$.
    If $f=1$ or $r = 0$, then $\phi(\HO{S}) = \phi(\relint(S)) +
    \phi(\relint(F)) \succeq 0$ by (ii). For $f > 1$, let 
    $F \subset \HO{S}$ be a half-open facet. Then $T = \HO{S} \setminus F$ is a
    half-open simplex with $f(T) < f$ and, by induction on $f$ and $r$, we get
    \[
        \phi(\HO{S}) \ = \ \phi(T) + \phi(F) \ \succeq \ 0.
    \]

    (iii) $\Rightarrow$ (i): Let $P\subseteq Q$ be two $\LL$-polytopes with
    $r-1 = \dim P = \dim Q$. As in the proof of Theorem \ref{thm:h-nonneg}, we
    can choose a dissection $Q = T_1 \cup T_2 \cup \cdots \cup T_N$ of $Q$
    into $(r-1)$-dimensional $\LL$-simplices such that $P = T_{M+1} \cup
    T_{M+2} \cup \cdots \cup T_N$ for some $M < N$. For a point $q \in \relint
    P$ general with respect to all $T_i$, it follows from
    Corollary~\ref{cor:h-add} that
    \[
        h^\phi_i(Q) - h^\phi_i(P) \ = \ 
        h^\phi_i(\half{q}{T_1}) + \cdots + h^\phi_i(\half{q}{T_M}).
    \]
    It is thus sufficient to show 
    \[
        h^\phi_i(\HO{S}) \succeq 0
    \]
    for any \emph{proper half-open} $\LL$-simplex $\HO{S}$, that is, $\HO{S} =
    \half{q}{S}$ for some general $q \not\in S$. We will show, that the
    corresponding partly open hypersimplex $\Pi_i=\Pi_i(\HO{S})$ can be
    dissected into half-open simplices. By a change of coordinates, we can
    assume that $S = \{ x \in V : x \ge 0\}$, where $V
    = \{ x \in \R^r : x_1 + \cdots + x_r = 1\}$, and 
    \[
        \HO{S} \ = \ \{ x \in S : x_j > 0 \text{ for } j \in I \}
    \]
    with $I = I_q(S) \neq \emptyset$. We can also assume that the general
    point $q \in V$ satisfies $q_j > 1$ for $j \not\in I$. The corresponding
    $i$-th partly-open
    hypersimplex is 
    \[
        \Pi_i \ = \ \{ x \in i\cdot V : x_j > 0 \text{ for } j \in I, x_j < 1
        \text{ for } j \not\in I \} \ = \ \half{q'}{\overline{\Pi}_i}
    \]
    where $q' = i \cdot q$. Hence, $\Pi_i$ is a half-open polytope and
    after choosing a dissection $\overline{\Pi}_i = S_1 \cup \cdots \cup S_l$ into
    simplices, we obtain from Lemma~\ref{lem:ho-decomp}
    \[
    \Pi_i \ = \ 
    \half{q'}{S_1} \cup
    \cdots \cup
    \half{q'}{S_l}
    \]
    and thus,
    \[
        \phi(\Pi_i) \ = \ \sum_{l=1}^k \phi (H_{q'} (S_k))  \ \succeq \ 0.
        \qedhere
    \]

\end{proof}

A $\LL$-valuation is \Defn{monotone} if $\phi (P) \preceq \phi (Q)$ for all
$\LL$-polytopes $P \subseteq Q$ and \Defn{nonnegative} if $\phi (P) \succeq 0$
for all $P\in \PolL$.  Clearly, every monotone valuation is nonnegative
but the converse is in general not true as the following example shows. 

\begin{example}\label{ex:d2}
    For $\LL = \Z^2$, define the $\Z^2$-valuation  $b(P) := \dvol(P) -
    \vol_2(P) - \chi(P)$. If $\dim P \le 1$, then $b(P) = \vol_1(P)$.
    For $\dim P = 2$, $2b(P) = | \partial P \cap \Z^2|$.
     This is clearly a nonnegative 
    valuation. But as the
    following figure shows $b$ is not monotone.
    \begin{center}
        \includegraphics[width=3.5cm]{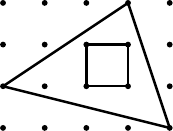}
    \end{center}
\end{example}

We call a $\LL$-valuation \Defn{weakly monotone} if $\phi(\{0\})\succeq 0$ and
$\phi(P) \preceq \phi (Q)$ for all $\Lambda$-polytopes $P \subseteq Q$ with $\dim
(P)=\dim(Q)$. It turns out, that monotonicity and weak monotonicity are in
fact equivalent.
\begin{prop}\label{prop:weakmonotonicity}
    Let $\phi$ be a $\LL$-valuation. Then $\phi$ is monotone if and only if
    $\phi$ is weakly monotone.
\end{prop}
\begin{proof}
    For $\LL$-polytopes  $P\subseteq Q$ we construct a chain of
    $\LL$-polytopes
    \[
        P=P_0\subseteq P_1\subseteq \cdots \subseteq P_r\subseteq Q,
    \]
    where $P_{i+1}=\conv(P_i\cup q_i)$ for some $q_i\in (Q\cap \LL)\setminus
    \aff(P_i)$ for all $0\leq i\leq r-1$, and $\dim(P_r)=\dim(Q)$. Hence, it
    suffices to prove that $\phi(P)\preceq \phi (Q)$ when $Q$ is a pyramid over
    $P$ with apex $a = 0$. If $P=\emptyset$, then $Q=\{0\}$ and
    $\phi(Q)\succeq
    0$ by definition. If $\dim(P) \geq 0$, then the truncated pyramid $T : =
    2Q\setminus (Q\setminus P)$ is contained in $2Q$ and is of equal
    dimension. Therefore
    \[
        0 \preceq \phi(2Q) -\phi(T) \ = \ \phi (Q)-\phi(P). \qedhere
    \]
\end{proof}

The next result gives us the relation to monotone valuations.
\begin{prop}\label{prop:monotone}
    Let $\phi$ be a weakly $h^*$-monotone $\LL$-valuation. Then $\phi$ is
    monotone.
\end{prop}
\begin{proof}
    We have to show that $\phi(P)\leq \phi(Q)$ for $\LL$-polytopes $P\subseteq
    Q$. By Proposition~\ref{prop:weakmonotonicity} we may assume that
    $\dim(P)=\dim(Q)$.  Let $Q = T_1 \cup T_2 \cup \cdots \cup T_N$ be a
    dissection of $Q$ into $\LL$-simplices such that $P = T_{M+1} \cup T_{M+2}
    \cup \cdots \cup T_N$ for some $M < N$. For a point $q\in \relint P$
    general with respect to all $T_i$ we obtain
    \[
        \phi(Q)-\phi (P) \ = \ \sum_{i=1}^M \phi(H_q T_i) \ \succeq \ 0
    \]
    by Theorem \ref{thm:h-weaknonneg}.
\end{proof}

The converse, however, is not true.

\begin{example}\label{ex:notweaklyhmonotone}
    Let $R$ be the lattice triangle with vertices $a = \binom{0}{0},b =
    \binom{2}{0},c = \binom{2}{1}$. Consider the valuation $\Mink{\dvol}{Q}$
    where $Q = [(0,0),(1,1)]$. It is easy to see that
    $\Mink{\dvol}{Q}$ is monotone. To see that $\Mink{\dvol}{Q}$ is not weakly
    $h^*$-monotone, we appeal to Theorem~\ref{thm:h-weaknonneg} and compute
    for the facet $F = \conv(b,c)$
    \[
        \Mink{\dvol}{Q}(\relint R) + \Mink{\dvol}{Q}(\relint F) \ = \ (-1) + 0
        \ < \ 0.
    \]
\end{example}

We close this section by summarizing the various relationships in the
following diagram:
\begin{center}
    \begin{tabular}{c}
        $h^{\ast}$-nonnegative\\ 
        $\Updownarrow$\\
        $h^{\ast}$-monotone
    \end{tabular}
    \ $\Longrightarrow$ \
    \begin{tabular}{c}
        weakly\\ $h^{\ast}$-monotone
    \end{tabular}
    \ $\Longrightarrow$ \
    \begin{tabular}{c}
        monotone\\
        $\Updownarrow$\\
        weakly monotone
    \end{tabular}
    \ $\Longrightarrow$ \
    nonnegative
\end{center}

\section{Cones of combinatorially positive valuations}  \label{sec:cones}

Let us assume that $G$ is a finite-dimensional $\R$-vector space. Then 
\[
    \Vals(\LL,G) \ = \ \{ \phi : \PolL \rightarrow G \text{
    $\LL$-valuation}\}
\]
inherits the structure of a real vector space.  Let $C \subset G$ be a closed
and pointed convex cone. Then we can define a partial order on $G$ by
\[
    x \ \preceq_C \ y \quad :\Longleftrightarrow \quad  y-x \in C.
\]
This partial order is compatible with the group structure on $G$ and $C = \{ x
\in G : x \succeq 0\}$. Throughout this section, $G$ will be partially ordered
by some $C$.
    
We write $\ValsCP(\LL,G)$ for the collection of combinatorially positive
$\LL$-valuation $\phi : \PolL \rightarrow G$. Observing that condition (iii)
in Theorem~\ref{thm:h-nonneg} is linear in $\phi$ shows that 
$\ValsCP(\LL,G)$ has typically a nice structure.

\begin{prop}
    The set $\ValsCP(\LL,G)$ is a convex cone.
\end{prop}

In the following sections we will study the geometry of this cone for $\LL =
\R^d$ and $\LL = \Z^d$.

\subsection{$\R^d$-valuations}

Our main result for $\LL = \R^d$ gives a precise description of
$\ValsCP(\R^d,G)$.

\begin{thm}\label{thm:conevol}
    Let $G$ be a finite-dimensional real vector space partially ordered by a
    closed and pointed convex cone $C$. Then
    \[
        \ValsCP(\R^d,G) \ \cong \ C.
    \]
    The isomorphism takes $c$ to $c \, \vol_d$.
\end{thm}

If $\dim G = 1$ and hence up to isomorphism $G = \R$ with the  usual order, we
obtain a new characterization of the volume.
\begin{cor}\label{cor:charvol}
    The volume is, up to scaling, the unique real-valued
    combinatorially positive \mbox{$\R^d$-valuation}.
\end{cor}

As a first step towards a proof of Theorem~\ref{thm:conevol}, we recall the
following result of McMullen.
\begin{thm}[{\cite[Theorem 8]{mcmullenEuler}}]\label{thm:monotonecont}
    Every monotone $\R^d$-valuation $\phi : \Pol{\R^d} \rightarrow \R$
    is continuous with respect to the Hausdorff metric.
\end{thm}

Since every combinatorially positive valuation is monotone
(Proposition~\ref{prop:monotone})
we conclude that the cone $\ValsCP(\R^d,G)$ is indeed a closed convex cone.
We recall the following well-known result; see, for example,
Gruber~\cite[Chapter 16]{gruber2007convex}).

\begin{lem}\label{lem:volsimple}
    If $\phi : \Pol{\R^d} \rightarrow \R$ is a simple, monotone
    $\R^d$-valuation, then $\phi = \lambda \vol_d$ for some $\lambda \ge 0$.
\end{lem}

\begin{proof}[Proof of Theorem \ref{thm:conevol}]
    Let $\phi$ be a combinatorially positive valuation. We will show that for
    every linear form $\ell : G \rightarrow \R$ that is nonnegative on $C$, the
    real-valued $\R^d$-valuation $\ell \circ \phi$ is a nonnegative multiple
    of the volume.  Since $C$ is pointed, this then proves $\phi = c
    \vol_d$ for $c = \phi([0,1]^d) \in C$.

    Since $\ell \ge 0$ on $C$, $\ell \circ \phi$ is monotone and by
    Theorem~\ref{thm:monotonecont} continuous in the Hausdorff metric.
    In light of Lemma~\ref{lem:volsimple} it thus suffices to prove that
    $\phi$ is simple.
    
    For every polytope $P \in \Pol{\R^d}$ let $g(P)=(g_0 (P),g_1 (P),\dots,
    g_d (P)) \in G^{d+1}$ be such that
    \[
        \sum_{n \ge 0} \phi(nP) t^n \ = \ 
        \frac{g_0(P)+g_1(P) t +\cdots + g_d(P) t^d}{(1-t)^{d+1}}.
    \]
    We denote the numerator polynomial by $g_P(t)$. For all $0 \le i \le d$,
    every $g_i$ is a continuous $\R^d$-valuation. If $\dim P =
    r$, then
    \[
        g_P(t) \ = \ (1-t)^{d - r} \sum_{i=0}^r h_i^\phi(P) t^i.
    \]
    In particular, if $\dim P = d$, then $g_i(P) = h^\phi_i(P) \in C$ for all
    $0 \le i \le d$.

    Now let $Q$ be of dimension $r < d$. Consider the sequence of
    polytopes \mbox{$Q_n = Q + \frac{1}{n} [0,1]^d$}. Then $\dim Q_n = d$ for all $n
    \ge 1$ and $h^\phi_i(Q_n) = g_i(Q_n) \rightarrow g_i(Q)$ for $n
    \rightarrow \infty$. Since $C$ is closed, we have $g_i(Q) \in C$ for all
    $i$. On the other hand, $(1-t) | g_Q(t)$ and therefore $\sum_{i=0}^d g_i
    (Q)=0$.  Since $C$ is pointed, we conclude that $g_i(Q)=0$ for all $i$ and
thus $\phi(Q) =
    0$.
\end{proof}

Using similar techniques, we can describe the cone
\[
    \ValsWCP(\R^d,G) \ := \ \{ \phi : \Pol{\R^d} \rightarrow G \text{ weakly
    $h^*$-monotone} \}.
\]

\begin{thm}\label{thm:wcp_real}
    Let $G$ be a finite-dimensional real vector space partially ordered by a
    closed and pointed convex cone $C$. Then
    \[
        \ValsWCP(\R^d,G) \ \cong \ C \times C.
    \]
    The isomorphism takes $(c_1,c_2)$ to $c_1 \chi + c_2 \vol_d$.
\end{thm}

\begin{proof}
    Proposition~\ref{prop:monotone} shows that weakly $h^*$-monotone implies
    monotone. It follows that for
    $c_1 := \phi(\{0\}) \in C$, 
    \[
        \psi : =\phi  - c_1 \chi
    \]
    is still a weakly $h^*$-monotone $\R^d$-valuation and, in particular,
    monotone.  Analogous to the proof of Theorem~\ref{thm:conevol}, we show
    that $\psi$ is simple and conclude that $\psi = c_2 \vol_d$ for some $c_2
    \in C$.

    Let $P \subseteq Q$ be two polytopes of dimension $r < d$.
    Consider the $d$-polytopes $P_n := P + \frac{1}{n}[0,1]^d$ and $Q_n := Q +
    \frac{1}{n} [0,1]^d$.  Then $\dim (P_n)=\dim (Q_n)=d$ and $P_n\subseteq
    Q_n$ for all $n\geq 1$. Following the proof of Theorem~\ref{thm:conevol}, 
    we infer that $g_{Q_n}(t) - g_{P_n}(t)$ has all coefficients in $C$ and that
    \[
        g_{Q_n}(t) - g_{P_n}(t) \xrightarrow{n \rightarrow \infty} g_{Q}(t) -
        g_{P}(t).
    \]
    However, since $\dim(P) = \dim(Q) < d$, $g_P(1) - g_Q(1) = 0$. Since $C$
    is pointed, this implies that $g_P(t) = g_Q(t)$ and $\psi(P) = \psi(Q)$.

    Let us assume that $0 \in P$. Then $P \subseteq nP$ for all $n \ge 1$ and
    hence $\psi(nP) = c$ for all $n \ge 1$. In particular $\psi(0P) =
    \psi(\{0\}) = c$ which implies that $\psi(P) = 0$.
\end{proof}

\begin{cor}\label{cor:Steiner}
    The Steiner valuation $S(P) =\vol _d (P+B_d)$ is not weakly
    $h^\ast$-monotone for $d>1$.
\end{cor}
\begin{proof}
    The quermassintegrals are linearly independent $\R^d$-valuations with
    $W_0$ being the volume and $W_d$ proportional to the Euler characteristic.
    Hence the representation~\eqref{eqn:Steinerpoly} shows that for $d>1$, $S$ is
not in
    the cone spanned by $\chi$ and $\vol_d$.
\end{proof}

It is known (cf.~\cite{gruber2007convex}) that the quermassintegrals are
nonnegative and monotone with respect to inclusion. Hence, using Hadwiger's
characterization result, the cone of nonnegative and the cone of monotone
rigid-motion invariant continuous valuations on convex bodies in $\R^d$ coincide
and are
isomorphic to $\R^{d+1}_{\ge 0}$. Meanwhile, the corresponding cones of
rigid-motion invariant (weakly) $h^*$-monotone valuations are still given by
Theorems~\ref{thm:conevol} and~\ref{thm:wcp_real}.

\subsection{Lattice-invariant valuations} 

Let $\LL$ be a lattice of rank $d$ that, without loss of generality, we can assume to be
$\Z^d$. A valuation $\phi : \Pol{\Z^d} \rightarrow G$ is
\Defn{lattice-invariant} if $\phi(T(P)) = \phi(P)$ for all $P \in \Pol{\Z^d}$
and every affine map $T$ with $T(\Z^d) = \Z^d$. A fundamental result on the
structure of lattice-invariant valuations was obtained by Betke and
Kneser~\cite{BetkeKneser}.  For $0 \le i \le d$, we define the \Defn{$\mathbf
i$-th standard simplex} as $\Delta_i := \conv\{0,e_1,\dots,e_i\}$, where
$\{e_1,\dots,e_d\}$ is a fixed basis for $\LL$.

\begin{thm}[{Betke--Kneser~\cite{BetkeKneser}}]\label{thm:betke_kneser}
    For every $a_0,a_1,\ldots,a_d \in G$ there is a unique lattice-invariant
    valuation $\phi : \Pol{\Z^d} \rightarrow G$ such that
	\[
        \phi(\Delta_i) = a_i \quad \text{ for all } 0 \le i \le d.
    \]
\end{thm}

In particular, there are lattice-invariant valuations $\phi_0,\dots,\phi_d :
\Pol{\Z^d} \rightarrow \Z$ such that $\phi_j(\Delta_i) = \delta_{ij}$ and every
valuation $\phi : \Pol{\Z^d} \rightarrow G$ admits a unique presentation as
\begin{equation}\label{eqn:phi_unique}
    \phi \ = \ 
    \phi(\Delta_0) \phi_0 \ + \  \cdots \ + \ \phi(\Delta_d) \phi_d.
\end{equation}

This implies that
\[
    \bVals( \Z^d, G) \ := \ \{ \phi : \Pol{\Z^d} \rightarrow G : \phi \text{
    lattice invariant}\} \ \cong \ G^{d+1}.
\]

We assume that $G$ is a real vector space of finite dimension,  partially
ordered by a closed and pointed convex cone $C$.  In this section we study the
cone of combinatorially positive, lattice-invariant valuations 
\[
    \bValsCP({\Z^d},G) \ := \ \ValsCP({\Z^d},G) \cap \bVals({\Z^d},G).
\]
In contrast to the case of (rigid-motion invariant) $\R^d$-valuations, this is
a proper convex cone.

\begin{prop} 
    The cone $\bValsCP({\Z^d},G)$ is of full dimension $(d+1) \cdot \dim C$.
\end{prop}

\begin{proof}
    For $\ell=1,\dots,d+1$, define the valuation $\dvol^\ell(P) :=
    \dvol(\ell\cdot
    P)$. Then $\dvol^\ell$ is lattice invariant and
    \[
        \dvol^\ell(\relint (P))  \ = \ \dvol(\relint(\ell\cdot P))   \ \ge \ 0
    \]
    shows that $\dvol^\ell$ is combinatorially positive.  Moreover, 
    $\dvol^1,\dots,\dvol^{d+1}$ are linearly independent. Indeed, assume that
    $\alpha_1\dvol^1 + \cdots + \alpha_{d+1} \dvol^{d+1} = 0$.  We have
    $\dvol^\ell(n[0,1]^d) = (\ell n + 1)^d$ and 
    \[
        \alpha_1 (n+1)^d + \alpha_2 (2n + 1)^d + \cdots + \alpha_{d+1}((d+1)n
        + 1)^d \ = \ 0
    \]
    for all $n$ implies $\alpha_i = 0$ for all $i$.
    
    Now, let $m = \dim C$ and let $c_1,\dots,c_m \in C$ be linearly
    independent. The lattice-invariant valuations $\{c_i E^{\ell} : 1 \le i \le m, 1 \le \ell
    \le d+1\}$ are linearly independent and  combinatorially positive which
    proves the claim.
\end{proof}

We will give a detailed description of $\bValsCP({\Z^d},G)$ that complements the
Betke--Kneser theorem.

\begin{thm}\label{thm:bValsCP}
    A lattice-invariant valuation $\phi : \Pol{\Z^d} \rightarrow G$ is
    combinatorially positive if and only if $\phi(\relint(\Delta_i)) \succeq 0$
    for all standard simplices $\Delta_i$, $i=0,\dots,d$. In particular,
    \[
        \bValsCP(\Z^d,G) \ \cong \ C^{d+1}.
    \]
\end{thm}

The theorem is equivalent to 
\begin{equation}\label{eqn:bValsCP_simplicial}
    \bValsCP(\Z^d,G) \ = \ \{ \phi \in \bVals({\Z^d},G) : \phi(\relint \Delta_i)
    \succeq 0 \text{ for all } i=0,\dots,d \}.
\end{equation}
The inclusion `$\subseteq$' follows from Theorem~\ref{thm:h-nonneg}(iii).  To
prove the reverse inclusion it is sufficient to show that every
lattice-invariant valuation $\phi$ is combinatorially positive if
$\phi(\relint(\Delta_i)) \succeq 0$ for all $i =0,\dots,d$. In dimensions $d
\le 2$, this is true since every lattice polytope can be triangulated into
unimodular simplices. In dimension $d=3$, a direct approach uses the
classification of empty lattice simplicies due to Reznick~\cite[Corollary
2.7]{Reznick86} and induction on the lattice volume similar to
Betke--Kneser~\cite{BetkeKneser}. 

Our proof of Theorem~\ref{thm:bValsCP} pursues a different strategy: Since the
right-hand side of~\eqref{eqn:bValsCP_simplicial} is a polyhedral cone, it is
sufficient to verify it is generated by a set of combinatorially positive
valuations. For the case $(G,C) = (\R,\R_{\ge0})$, such generators will be
given in the next section.

\section{A discrete Hadwiger theorem}\label{sec:hadwiger}

Hadwiger's characterization theorem~\cite{Hadwiger} states that every continuous
rigid-motion invariant valuation $\phi$ on convex bodies in $\R^d$ is uniquely
determined by the evaluations $(\phi(S_i))_{i=0,\dots,d}$ where $S_0,\dots,S_d
\subset \R^d$ are arbitrary but fixed convex bodies with $\dim S_r = r$.  
From this it is easy to deduce that the quermassintegrals $W_i$, i.e. the
coefficients of Steiner polynomial
\[
    \vol(t K + B_d) \  = \ \sum_{i=0}^d \binom{d}{i}W_{d-i}(K)\, n^i
\]
are linearly independent and hence span the space of continuous rigid-motion
invariant valuations. The quermassintegral $W_i$ is homogeneous of degree $d-i$
and hence up to scaling $W_0,\dots,W_d$ is the unique homogeneous basis for
this space.

The Betke--Kneser result (Theorem~\ref{thm:betke_kneser}) is a natural discrete
counterpart: Every lattice-invariant valuation $\phi : \Pol{\Z^d} \rightarrow
G$ is uniquely determined by its values on $d+1$ lattice simplices of
different dimensions. A homogeneous basis for the space of lattice-invariant
valuations is given by the coefficients of the Ehrhart polynomial
\[
    \dvol_P(n) \ = \ \ e_d(P) n^d + \cdots + e_0(P).
\]

However, there are many desirable properties of quermassintegrals that the
valuations $e_i$ lack.  As they are special mixed volumes, the
quermassintegrals are nonnegative and monotone. These properties distinguish
them from all other basis for the space of rigid-motion invariant valuations:
The cones of nonnegative and, equivalently, monotone rigid-motion invariant
valuations are spanned by the quermassintegrals. Unfortunately, the valuations
$e_i$ are neither monotone nor nonnegative; cf.~\cite[Chapter 3]{BR}.  This
was Stanley's original motivation for the $h^*$-monotonicity
result~\cite{Stan93} given in Corollary~\ref{cor:Ehr_NN}. In this section we
study a basis for $\bVals(\Z^d,\Z)$ that is combinatorially positive and hence
by the results of Section~\ref{sec:weak} also nonnegative and monotone. This
yields a discrete Hadwiger Theorem. 

In a different binomial basis Ehrhart's result~\eqref{eqn:hE-vec} states that 
\begin{equation}\label{eqn:disc_Had}
    \dvol_P(n) \ = \ f_0^* (P)\binom{n-1}{0}  + f_1^*
    (P)\binom{n-1}{1}  +  \cdots  +   f_d^*(P)\binom{n-1}{d}.
\end{equation}
for some $f_i^*(P) \in \Z$. These coefficients take the role of the
quermassintegrals for combinatorial positivity.

\begin{thm}\label{thm:disc_Had}
    Let $\phi : \Pol{\Z^d} \rightarrow \R$ be a lattice-invariant valuation. Then
    $\phi$ is combinatorially positive if and only if 
    \[
        \phi \ = \ \alpha_0 f_0^* + \alpha_1 f_1^* +  \cdots + \alpha_d f_d^*
    \]
    for some $\alpha_0,\dots,\alpha_d \ge 0$.
\end{thm}

Since $\binom{n-1}{0},\dots,\binom{n-1}{d}$ is a basis for univariate polynomials
of degree $\le d$, the valuations $f_0^*,\dots,f_d^*$ are a basis for
$\bVals(\Z^d,\R)$. The following lemma gives an explicit expression of $\phi$
in terms of this basis.

\begin{lem}\label{lem:f_star_rep}
    For all $i,j = 0,1,\dots,d$
    \[
        f_j^*(\relint(\Delta_i)) \ = \ \delta_{ij}.
    \]
    In particular,
    for every lattice invariant valuation $\phi \in \bVals(\Z^d,G)$
    \[
        \phi \ = \ 
        \phi(\relint(\Delta_0))\, f_0^* \ + \ 
        \phi(\relint(\Delta_1))\, f_1^* \ + \ 
        \cdots \ + \
        \phi(\relint(\Delta_d))\, f_d^*.
    \]
\end{lem}
\begin{proof}
    For the first claim, we simply note that $\dvol_{\relint(\Delta_i)}(n)
    = \binom{n-1}{i}$. For the second claim, observe that if
    $\phi(\relint(\Delta_i)) = a_i$ for all $i=0,\dots,d$, then
    \eqref{eqn:sumrelint} together with the fact that every $r$-face of
    $\Delta_i$ is lattice isomorphic to $\Delta_r$ yields
    \[
        \phi(\Delta_i) \ = \ \sum_{r=0}^i \binom{i+1}{r+1} a_r.
    \]
    By Theorem~\ref{thm:betke_kneser}, there is a unique valuation taking
    these values on standard simplices and~\eqref{eqn:phi_relint} finishes the
    proof.
\end{proof}

Thus, if $\phi$ is combinatorially positive, then $\alpha_i =
\phi(\relint(\Delta_i)) \ge 0$ which proves necessity in
Theorem~\ref{thm:disc_Had}. For sufficiency, we need to show that $f_j^*$ is
combinatorially positive for all $j$. That is, we need to show that
$f_j^*(\relint \Delta) \ge 0$ for all lattice simplices $\Delta$. 

For a lattice polytope $P \in \Pol{\Z^d}$, $f^*(P) =
(f_0^*(P),\dots,f_d^*(P))$ is called the \emph{$f^*$-vector}. The
\mbox{$f^*$-vector} was introduced and studied by Breuer~\cite{breuer}. He
showed that $f_j^*(\relint(P)) \ge 0$ and gave an enumerative
interpretation for lattice simplices. We deduce the nonnegativity
result from more general considerations. For a translation-invariant valuation
$\phi: \PolL \rightarrow G$, where $\LL$ is not restricted to lattices, we
define its \Defn{$f^*$-vector} $f^\phi = (f_0^\phi,\dots,f_d^\phi)$ such that
for every $P \in \PolL$ 
\[
    \phi_P(n) \ = \ \sum_{i=0}^d f_i^\phi(P) \binom{n-1}{i}
\]
for all $n \ge 0$. Equivalently, $f_i^\phi$ is given by
\[
    f_i^\phi(P) \ := \ \sum_{k=0}^i \binom{i}{k} (-1)^{i-k} \phi((k+1)P)
\]
Notice the $f_i^\phi$ are translation-invariant $\LL$-valuations. 

\begin{thm}\label{thm:f_star_equivalent}
    Let $\LL \subset \R^d$ be a lattice or a finite-dimensional vector
    subspace over a subfield of $\R$ and $G$ a partially ordered abelian
    group.  For a $\LL$-valuation $\phi\colon \PolL \rightarrow G$ the
    following are equivalent:
    \begin{itemize}
        \item[(i)] $\phi$ is combinatorially positive.
        \item[(ii)] $f_i ^\phi$ is combinatorially positive for all $i=0,\ldots 
        d$.
   \end{itemize} 
\end{thm}
\begin{proof}
    For the implication $(ii)\Rightarrow (i)$ simply observe that 
    \[
        \phi(\relint(P)) \ = \ \phi_{\relint(\Delta)}(1) \ = \ f^\phi_0
        (\relint(P)) \ge  0
    \]
    for all $P \in \PolL$. The claim now follows from
    Theorem~\ref{thm:h-nonneg}.

    For $(i)\Rightarrow (ii)$, we claim that 
    \[
        f^\phi _{r-k}(\relint(-P)) \ = \ 
        \sum_{i=k}^r 
        h^\phi_i(P)
        \binom{i}{k}
    \]
    for any $r$-dimensional $\LL$-polytope $P$. Assuming that $\phi$ is
    $h^*$-nonnegative then shows combinatorial positivity of $f_i^\phi$.
    To prove the claim, we use Theorem~\ref{thm:reciprocity} together with the
    identity
    $(-1)^r\binom{-n+r-k}{r} = \binom{n-1+k}{r}$ to get
    \[
        \phi_{\relint(-P)}(n)  \ = \ (-1)^r \phi_P(-n) \ = \
        h^\phi_0 (P) \binom{n-1}{r} + 
        h_1^\phi(P) 
        \binom{n}{r} + \cdots 
        + h_r^\phi(P) \binom{n-1+r}{r}
    \]
    and collecting terms completes the proof.
\end{proof}

To complete the proof of Theorem~\ref{thm:disc_Had}, we use Stanley's
nonnegativity of the $h^*$-vector (Corollary~\ref{cor:Ehr_NN}) together with
Theorem~\ref{thm:f_star_equivalent}. The same reasoning also yields a proof of
Theorem~\ref{thm:bValsCP}. 

\begin{proof}[Proof of Theorem~\ref{thm:bValsCP}]
    The map $\Psi : \bVals({\Z^d},G) \rightarrow G^{d+1}$ given by 
    \[
        \phi \ \mapsto (\phi(\relint \Delta_i))_{i=0,\dots,d}
    \]
    is an isomorphism by Lemma~\ref{lem:f_star_rep}. In particular $\Psi$
    takes $\bValsCP({\Z^d},G)$ into $C^{d+1}$. To show that this is a surjection,
    we use Theorem~\ref{thm:f_star_equivalent} to deduce that for every
    $a = (a_0,\dots,a_d) \in C^{d+1}$ the valuation
    \[
        \phi \ = \ a_0 f_0^* + \cdots + a_d f_d^*
    \]
    is combinatorially positive with $\Psi(\phi) = a$.
\end{proof} 
    
It turns out that there is also a Hadwiger-type result for weakly
$h^*$-monotone valuations. For this consider the Ehrhart polynomial in the
basis
\newcommand\wfs{\widetilde{f}^*}
\begin{equation}\label{eqn:fbar_basis}
    E_P(n) \ = \ \wfs_0(P) \binom{n}{0} \ + \ \cdots \ + \ \wfs_0(P)
    \binom{n}{d}.
\end{equation}

\begin{thm}\label{thm:disc_Had_weak}
    A lattice-invariant valuation $\phi : \Pol{\Z^d} \rightarrow \R$ is
    weakly $h^*$-monotone if and only if 
    \[
        \phi \ = \ \alpha_0 \wfs_0 + \alpha_1 \wfs_1 +  \cdots +
        \alpha_d \wfs_d
    \]
    for some $\alpha_0,\dots,\alpha_d \ge 0$.
\end{thm}        

As for the proof of Theorem~\ref{thm:disc_Had}, the crucial observation is
that $\phi$ is weakly $h^*$-monotone if and only if an analogous extension
$\widetilde{f}^\phi_i$ is weakly $h^*$-monotone for all $i$. Necessity follows from
the proof of Theorem~\ref{thm:h-weaknonneg} where it is shown that if $\phi$
is weakly $h^*$-monotone then $h^*_i(\HO{S}) \succeq 0$ for all proper half-open
simplices $\HO{S}$.

\subsection{Dimension $d=2$} \label{ssec:d2}
In this section we study in detail the
cone $\bValsCP(\Z^2,\R)$ in relation to the cones
\begin{align*}
    \bValsM(\Z^2,\R) \ &:= \ \{ \phi \in \bVals(\Z^2,\R) : \phi(P) \ge
    \phi(Q) \text{ for lattice polytopes } Q \subseteq P \} \text{ and }\\
    \bValsN(\Z^2,\R) \ &:= \ \{ \phi \in \bVals(\Z^2,\R) : \phi(P) \ge 0
     \text{ for } P \in \Pol{\Z^2}\}.
\end{align*}

The results of Section~\ref{sec:weak} imply 
\[
    \bValsCP(\Z^2,\R) \ \subsetneq \ \bValsM(\Z^2,\R) \ \subsetneq \
    \bValsN(\Z^2,\R).
\]

We study these cones in the usual monomial basis.  From Pick's theorem
(cf.~\cite[Theorem 2.8]{BR}) the Ehrhart polynomial of a lattice polytope can
be expressed as
\[
    E_P(n) \ = \ \vol_2(P) n^2 + b(P) n + \chi(P),
\]
where $b(P)$ was introduced in Example~\ref{ex:d2}.  In particular, the
coefficients $\vol_2, b, \chi$ are lattice-invariant, nonnegative and
homogeneous of
degrees $2,1,0$, respectively.

\begin{prop}\label{prop:nonnegdim2}
    The cone $\bValsN$ is the simplicial cone generated by $\vol_2, b$
and $\chi$.
\end{prop}

From Theorem \ref{thm:bValsCP} we know that $\bValsCP$ is simplicial and
generated by 
\[
        \dvol = \vol_2+ b + \chi, \quad \vol_2, \quad 3 \vol_2 + b.
\]

Determining the cone of monotone valuations is harder since $b$, as opposed to
$\vol_2$ and $\chi$, is \emph{not} monotone; see Example~\ref{ex:d2}.

\begin{thm}\label{prop:monotonedim2}
    The cone $\bValsM$ is simplicial and generated by
    \[
    \chi, \quad  b + \vol_2, \quad  \vol_2.
    \]
\end{thm}
\begin{proof}
    First we observe that $b + \vol_2 = \dvol - \chi$ and hence the given
    valuations are indeed monotone.
    
    Now, let $\phi =\alpha \vol_2 +\beta b + \gamma \chi$ be a monotone
    translation-invariant valuation. Since $\phi$ is monotone, we have
    $\alpha,\beta,\gamma \ge 0$. We can assume that $\gamma = 0$ as $\phi -
    \phi(0)$ is still monotone.  Let $Q_n = [0,n]^2$ be the $n$-th dilated unit
    square and set  $P_n = \conv(Q_n \cup \{(-1,-1)\})$. 
    \begin{eqnarray*}
        \phi(Q_n) & = & \alpha n^2 + 2 \beta n, \text{ and }\\
        \phi(P_n) & = & \alpha (n^2 + n)+\beta (n+1), 
    \end{eqnarray*}
    By monotonicity, we obtain
    \[
        0 \ \leq \ \phi(P_n) - \phi(Q_n) \ = \ (\alpha -  \beta)n + \beta
    \]
    for all $n\geq 0$ and thus $\alpha \geq  \beta$. The cone generated by
    the inequalities $\alpha \geq 0$, $\gamma \geq 0$ and $\alpha \geq \beta$
    is generated by the rays $\vol_2,  \vol _2 + b ,$ and $ \chi$.
\end{proof}

In the space $\bVals(\Z^2,\R) = \{ \alpha \vol_2 + \beta b + \gamma \chi :
\alpha,\beta,\gamma \in \R\}$, a cross-section of the cones with
\mbox{$\{\alpha +
\beta + \gamma = 1\}$} is given in Figure~\ref{fig:crosssec}.

\begin{figure}[H]
\centering
\begin{tikzpicture}[scale=5]
    \path[draw, fill=yellow!20] 
        (  1.0000,  0.00000) node[right] {$\vol_2$} --
        (-0.50000,  0.86603) node[above] {$b$} --
        (-0.50000, -0.86603) node[below] {$\chi$} -- cycle;
    \path (-0.1, 0.1) node[above left] {$\bValsN$};

    \path[draw, fill=green!20] 
        (  1.0000,  0.00000) -- 
        ( 0.25000,  0.43301) node[above right] {$\tfrac{1}{2} \vol_2 +
                                            \tfrac{1}{2}b$} -- 
        (-0.50000, -0.86603) -- cycle;
    \path (0.25000, -0.14434) node[below left] {$\bValsM$};

    \path[draw, fill=red!20] 
        (0.00000, 0.00000) node[left] {$\tfrac{1}{3}\dvol$} -- 
        ( 1.0000, 0.00000) -- 
        (0.62500, 0.21651) node[above right] {$\tfrac{3}{4}\vol_2
                        + \tfrac{1}{4}b$}
        -- cycle;
    \path (0.54167, 0.072169) node {$\bValsCP$};
    \end{tikzpicture}
\caption{Cross-section of the nested cones $\bValsCP \subset
\bValsM \subset \bValsN$ for $\Lambda = \Z^2$.}\label{fig:crosssec}
\end{figure}

It would be very interesting to see if a Hadwiger-type result can be given for
monotone or nonnegative valuations. In the language of cones, we conjecture
the following.
\begin{conj}
    The cones of lattice-invariant valuations $\phi: \Pol{\Z^d} \rightarrow
    \R$ that are monotone or respectively nonnegative are simplicial.
\end{conj}

In dimension $d=2$, it can also be observed that the cone of lattice-invariant
monotone valuations coincides with the cone of weakly $h^*$-monotone
valuations. Example~\ref{ex:notweaklyhmonotone} shows that this is not true
without the restriction to \emph{lattice-invariant} valuations. We do not
believe that these cones coincide in general.  However, we currently do not
have a counterexample.

\nocite{jochemko2015combinatorics}
\bibliographystyle{siam}
\bibliography{CombAspectsValuations}

\end{document}